 \documentclass[a4paper, 12pt]{article}
\usepackage{color}

  \usepackage[normalem]{ulem}
\usepackage{chngcntr}
\usepackage{ulem}

\usepackage[colorlinks=true]{hyperref} %to link references and citations
\usepackage{soul}
\newcommand{\MnR}{\mathbb{R}^{n\times n}}
\newcommand{\Sn}{{\cal S}^n}

\newcommand{\R}{\mathbb{R}}
\newcommand{\Rn}{\mathbb{R}^{n}}

\newcommand{\Aut}{\mathrm{Aut}}
\newcommand{\Hn}{\mathbb{H}^n}
\usepackage{mathtools}
\usepackage{witharrows}
\usepackage{setspace}
\usepackage{amsfonts}
\usepackage{epsfig}
\usepackage{latexsym}
\usepackage{amsthm}
\usepackage{amsmath}
\usepackage{amssymb,amsfonts}
\usepackage{array}
\usepackage{nicematrix}
\usepackage{enumerate}
\usepackage{graphicx}
\usepackage{mathrsfs}
\usepackage{bm}
\usepackage{mathtools}
\usepackage{multirow}
\usepackage{multicol}
\usepackage{nicematrix}
\usepackage{arydshln}
\usepackage{blkarray}

\newtheorem{theorem}{Theorem}[section]
\newtheorem{corollary}{Corollary}[section]
\newtheorem{lemma}{Lemma}[section]
\newtheorem{definition}{Definition}[section]
\newtheorem{remark}{Remark}[section]

\newtheorem{example}{Example}[section]

\newtheorem{proposition}{Proposition}[section]
\counterwithin{figure}{section}

\usepackage[margin=1in]{geometry}
\usepackage{setspace}
\usepackage[english]{babel}
\usepackage{graphicx}
\usepackage{amsmath,amssymb,amsfonts}
\usepackage[normalem]{ulem}
\usepackage{fancyhdr}
\usepackage{xcolor}
\usepackage{url}

\usepackage{float}
\usepackage{tikz}
\usepackage{subcaption}
\usetikzlibrary{graphs, graphs.standard}
\usetikzlibrary{decorations.markings}

\usepackage{hyperref}
\hypersetup{
    colorlinks = true,
    linkcolor = black,
    anchorcolor = black,
    citecolor = blue,
    filecolor = black,
    urlcolor = black
    }
\usepackage[utf8]{inputenc}

\newcommand{\SOL}{\mathrm{SOL}}
\newcommand{\LCP}{\mathrm{LCP}}
\newcommand{\bfR}{{\bf R}}

\newcommand{\bfQ}{{\bf Q}}
\newcommand{\bfE}{{\bf E}}
\newcommand{\bfP}{{\bf P}}
\newcommand{\bfZ}{{\bf Z}}
\newcommand{\tr}{\mathrm{trace}}
\newcommand{\bfS}{{\bf S}}
\newcommand{\bfone}{\bf 1}
\newcommand{\V}{{\cal V}}
\newcommand{\E}{{\cal E}}

\newfont{\bb}{msbm10}

\title{Linear complementarity properties \\of some classes of banded matrices}

\author{}
\date{}
\author{Samapti Pratihar\thanks{Department of Mathematics, Indian Institute of Technology Madras,
		Chennai, 600036, India (ma22d019@smail.iitm.ac.in)} \and M. Seetharama Gowda\thanks{Department of Mathematics and Statistics, University of Maryland, Baltimore County\\Baltimore, Maryland 21250, USA (gowda@umbc.edu)} \and K.C. Sivakumar\thanks{Department of Mathematics, Indian Institute of Technology Madras,
		Chennai, 600036, India (kcskumar@iitm.ac.in)}}
\begin{document}
\maketitle

\begin{abstract}
 A banded matrix is a real square matrix where nonzero entries appear around the main diagonal. In this article, we consider linear complementarity properties of (variants) of banded matrices. Focusing on triangular matrices and the newly defined bidiagonal southwest matrices, we describe several results characterizing the $\bfQ$-property in terms of the sign patterns and determinant of the given matrix. As a byproduct, we describe all $\bfQ$-matrices of size $2\times 2$. Extending these results to Euclidean Jordan algebras, we consider matrix-based linear transformations and study the $\bfQ$-property. In particular, we show that a rank-one linear transformation of the form $a\otimes b$ has the $\bfQ$-property if and only if either $a>0,\,b>0$ or $a<0$, $b<0$.
\end{abstract}

\vspace{1cm}
\textbf{Keywords:} Linear complementarity problem $\cdot$ Triangular matrix $\cdot$ Bdsw matrix $\cdot$ Euclidean Jordan algebra $\cdot$ Symmetric cone LCP

\textbf{AMS Subject classifications:}  15B35, 15B99, 17C55, 90C33

\clearpage
%%%%%%%%%%%%%%%%%%%%%%%%%%%
\section{Introduction}
 Given an $n\times n$ real matrix $A$ and a vector $q\in \Rn$, the {\it linear complementarity problem}, $\LCP(A,q)$, is to find $x \in \Rn$ satisfying the following conditions:
\begin{equation} \label{lcp}
x\ge 0,\,
Ax + q \ge 0,\,\mbox{and}\,\, \\
x^\top (Ax + q)=0.
\end{equation}
This problem is a particular instance of a cone complementarity problem, which itself is a special case of a variational inequality problem. LCP and its generalizations have been well-studied in the literature with numerous applications to optimization, economics, mechanics, etc., see e.g., \cite{cottle-pang-stone, facchinei-pang}. In this paper, we focus on the above LCP as well as the symmetric cone LCP, formulated in the setting of an Euclidean Jordan algebra. In the LCP theory, one studies classes of matrices having some specific properties.
For example, the class $\bfQ$ consists of matrices $A$ for which $\LCP(A,q)$ has a solution for all $q\in \Rn$. (A matrix in this class is said to be a $\bfQ$-matrix or said to have the $\bfQ$-property.) While there is no simple characterization of matrices within this class, one studies certain tractable classes within $\bfQ$ such as the class $\bfR$, consisting of matrices $A$ for which $\LCP(A,0)$ and $\LCP(A,d)$ have only trivial/zero solutions for some $d>0$; another important subclass is the class $\bfP$, consisting of matrices $A$ with every principal minor positive. Given an arbitrary matrix, even these conditions are difficult to check. However, in a few instances, when the matrix $A$ has an easily recognizable structure, we can characterize the $\bfQ$-property. For example,
\begin{itemize}
    \item When $A$ is a nonnegative matrix, $A$ has the $\bfQ$-property if and only if the diagonal of $A$ is positive \cite{murty};
    \item When $A$ is a $\bfZ$-matrix (meaning that its off-diagonal entries are nonpositive), $A$ is a $\bfQ$-matrix if and only if there is a positive vector $u$ with $Au$ positive (equivalently, $A$ is a $\bfP$-matrix) \cite{cottle-pang-stone};
    \item When $A$ is of rank-one, $A$ is a $\bfQ$-matrix if and only if $A$ is a positive matrix (i.e., all its entries are positive) \cite{sivakumar-sushmitha-tsatsomeros2022}.
\end{itemize}
Motivated by the above, we ask if similar characterizations could be obtained for banded matrices and their variants. 
Recall that a {\it banded matrix} is a real square matrix where nonzero entries appear around the main diagonal. Examples include upper/lower triangular matrices, bi/tri-diagonal matrices, Hessenberg matrices (in particular, the newly defined bdsw matrices).  Such matrices appear in numerous areas, including matrix theory, numerical analysis, partial differential equations, control theory, etc. The present paper grew out of the question: What are the linear complementarity properties of these matrices?

By way of answering the above question, in the first part of the paper, we initiate a study of the LCP properties (especially the $\bfQ$-property) of banded matrices and their variants. To start with, we show that a {\it triangular matrix has the $\bfQ$-property if and only if its diagonal is positive}; we prove a similar result when $A$ is obtained from an upper triangular matrix by adjoining a suitable nonnegative row. Next, we consider a bidiagonal matrix \cite{higham} with an additional entry in the southwest corner - resulting in a `bidiagonal southwest' (bdsw) matrix \cite{pratihar-sivakumar2024}. In our paper, for $n\geq 2$, these are matrices of the form
\begin{equation} \label{bdsw matrix form}
A = \begin{bmatrix}
    a_{11} & a_{12} &  &  0&  & 0 \\
     & a_{22} & a_{23} &  &  &  \\
     &  & a_{33} & a_{34} &   &  \\
    0 &    &   & \ddots & \ddots & 0\\
     &  &  &  & a_{(n-1)(n-1)} & a_{(n-1)n} \\
    a_{n1} &  & 0 &  & 0 & a_{nn}
\end{bmatrix},
\end{equation}
where $a_{ij}\in \R$; If $a_{n1}=0$, we get a bidiagonal matrix. Interest in bdsw matrices comes from the observation that when all $a_{ij}$ are nonzero, the corresponding directed graph is a simple directed cycle with loops, and the inverse of a nonsingular $A$ has all its entries nonzero \cite{pratihar-sivakumar2024}. 
Exploiting the invariance properties of the set of all bdsw matrices (that is, by considering certain row-column changes, principal submatrices, Schur complements, etc.) and using topological degree results, we show that the $\bfQ$-property of a bdsw matrix $A$ depends on the signs of the entries $a_{ij}$ and/or the determinant. As a byproduct, we solve the problem of characterizing the $\bfQ$-property of matrices of size $2\times 2$.
While our results are for banded matrices and their variants, one can obtain non-banded $\bfQ$-matrices via simultaneous row-column permutations, principal pivot transformations, and deletion/addition of suitable rows/columns. Although our analysis in the paper is limited to triangular and bdsw matrices, we hope to use our techniques to study other banded matrices such as tridiagonal matrices and their variants.

Going beyond the (standard) linear complementarity problem (\ref{lcp}), in the second part of our paper, we consider the symmetric cone linear complementarity problem formulated in the setting of an Euclidean Jordan algebra. To elaborate, consider an Euclidean Jordan algebra $\V$ with its symmetric cone $\V_+$. (For example, $\V$ could be the space of all $n\times n$ real symmetric matrices with $\V_+$ denoting the cone of positive semidefinite matrices in $\V$.) Given a linear transformation $L$ on $V$ and $q\in \V$, we define the symmetric cone LCP, $\LCP(L,\V_+,q)$, as the problem of finding an $x\in \V$ such that 
\begin{equation}\label{symmetric cone lcp}
x\in \V_+,\,y:=L(x)+q\in \V_+,\,\mbox{and}\,\,\langle x,y\rangle=0.
\end{equation}
While such a general problem has been studied in numerous works, see e.g., \cite{gowda-song1999, gowda-sznajder-tao2004}, our objective here is to consider $L$ as a matrix-based linear transformation. Given a Jordan frame $\{e_1,e_2,\ldots, e_n\}$ in $\V$, we associate a linear transformation $\widehat{A}$ to each matrix $A=[a_{ij}]\in \MnR$:
$$\widehat{A}:=\sum_{i,j=1}^{n}a_{ij}e_i\otimes e_j,$$ where $(e_i\otimes e_j)(x):=\langle e_j,x\rangle\,e_i.$ Then, corresponding to any (cone) automorphism $\phi$ of $\V_+$, we consider $$L:=\phi^T\widehat{A}\phi.$$ By connecting the linear complementarity properties of $A$ and $\widehat{A}$, we show that under certain conditions, $A$ has the $\bfQ$-property relative to $\Rn_+$ (i.e., in the standard LCP setting) if and only if $L$ has the $\bfQ$-property relative to $\V_+$. This will allow us to extend the LCP properties of banded matrices studied in the first part to a general Euclidean Jordan algebra setting. However, this extension is not limited to banded matrices. By taking $A$ to be a rank-one matrix, we extend a recent result of \cite{sivakumar-sushmitha-tsatsomeros2022} (mentioned earlier) to a rank-one transformation over an Euclidean Jordan algebra. 

 An outline and brief summary of the results are as follows: In Section 2, we present some necessary preliminary material, specifically addressing the LCP concepts. In Section 3, we show that 
       a (variant of a) triangular matrix is in $\bfQ$ if and only if its diagonal is positive.
In Section 4, we deal with bidiagonal southwest (bdsw) matrices. Since a matrix with a nonpositive row cannot have the $\bfQ$-property, {\it we assume each row of a (bdsw) matrix under consideration has a positive entry.} With this assumption, we consider four types of bdsw matrices and characterize the $\bfQ$-property in each case. 
\begin{itemize}
    \item A Type-I bdsw matrix contains at least one nonnegative row. For such a matrix, we provide a characterization of the $\bf Q$-property in terms of the signs of diagonal or superdiagonal entries, see Theorems \ref{Bdsw, an1>=0, ann>0} $-$ \ref{bdsw with kth row nonnegative}.
    \item A Type-II bdsw matrix has positive diagonal entries with superdiagonal and southwest corner entries negative. We show that such a matrix has the $\bfQ$-property if and only if its determinant is positive, see Theorem \ref{type2 bdsw matrix theorem}.
    \item A Type-III bdsw matrix is the negative of a Type-II bdsw matrix. It has negative diagonal entries with superdiagonal and southwest corner entries positive. For such a matrix $A$, we show (in Theorem \ref{type3 theorem}) that $A$ has the $\bfQ$-property if and only if $(-1)^{n+1}\det A>0$. 
    \item A Type-IV bdsw matrix is a bdsw matrix that does not belong to any of the previous type-classes. In such a matrix, every row has one positive and one negative entry, and there are at least two rows, one row with a positive diagonal entry and another row with a negative diagonal entry.
    For such a matrix $A$, we show, in Theorem \ref{type4 theorem}, that the $\bfQ$-property is equivalent to $(-1)^{k+1}\det A >0$, where $k$ is the number of negative diagonal entries.
    \end{itemize}
A key observation, based on the notion of (topological) degree is the following:
{\it A bdsw matrix has the $\bfQ$-property if and only if it is an $\bfR_0$-matrix with degree $\pm 1.$}

In Theorem \ref{2*2 Q characterization}, {\it we provide a complete characterization of $\bfQ$-matrices of size $2\times 2$ in terms of sign patterns and determinant.} 

In our final section, we describe some complementarity results for matrix-based linear transformations on Euclidean Jordan algebras. In particular, we characterize the $\bfQ$-property of a rank-one linear transformation.

%%%%%%%%%%%%%%%%%%%%%%%%%%%%%%
\section{Preliminaries}
\subsection{Linear complementarity problems}
Throughout, we write $\Rn$ for the real Euclidean $n$-space with usual inner product. The nonnegative orthant is denoted by $\Rn_+$; we regard elements of $\Rn$ as column vectors and write $x\geq 0$ $(x>0)$ when $x\in \Rn_+$ (respectively, $x\in \Rn_{++}=\mathrm{interior}\,(\Rn_+)$). We use the bold letter $\bfone$ for the vector of $1$s in $\Rn$.\\
Note: When a (column) vector $x\in \Rn$ is decomposed in terms of two component vectors, say, $x_1\in \R^k$ and $x_2\in \R^{n-k}$, we write (for simplicity) $x=(x_1,x_2)$  instead of $x=(x_1^T,x_2^T)^T$.\\
Throughout, $\MnR$ denotes the set of all real $n\times n$ matrices. Any matrix in $\MnR$ is said to have order $n$. For a real square matrix $A$, we write $A\geq 0$ ($A>0$) to mean that all its entries are nonnegative (respectively, positive). We rewrite (\ref{lcp}) in the form
$$0\leq x\perp Ax+q\geq 0.$$

We recall the following definitions from \cite{cottle-pang-stone}.
\begin{definition}
    Let $A\in \MnR$. Then,
\begin{itemize}
\item [$(1)$] $A$ is said to be a $\bfQ$-matrix if for every $q\in \Rn$, $\LCP(A,q)$ has  a solution.
 \item [$(2)$] $A$ is an $\bfR_0$-matrix if $\LCP(A,0)$ has  only the trivial/zero solution.
\item [$(3)$] Let $d>0$ in $\Rn$. Then $A$ is an $\bfR(d)$-matrix if $\LCP(A,d)$ and $\LCP(A,0)$ have only trivial/zero solution. $A$ is said to be an $\bfR$-matrix if $A\in \bfR(d)$ for some $d>0$.
\item [$(4)$] $A$ is an $\bfE_0$-matrix $(\bfE$-matrix$)$ if for every $d>0$ $($respectively, $d\geq 0)$, zero is the only solution of $\LCP(A,d)$. Equivalently, $A\in \bfE_0\,\,(\bfE)$ if for all $0\neq x\geq 0$, $\max\limits_{x_i\neq 0}\, x_i(Ax)_i \geq 0\, (>0),$ see \cite{cottle-pang-stone}, Theorems 3.9.3 and 3.9.11.
\item [$(5)$] $A$ is an $\bfR^*$-matrix if $A$ is both an $\bfR_0$-matrix and an $\bfE_0$-matrix. 
\item [$(6)$] $A$ is an $\bfS$-matrix if there exists $x>0$ such that $Ax>0$.
\item [$(7)$] $A$ is a $\bfP_0$-matrix if all principal minors of $A$ are nonnegative. Such a matrix is known to be an $\bfE_0$-matrix (\cite{cottle-pang-stone}, Theorems 3.4.2 and 3.9.3).
\item [$(8)$] $A$ is a $\bfP$-matrix if all principal minors of $A$ are positive. 
\item [$(9)$] $A$ is a $\bfQ_0$-matrix if the $\LCP(A,q)$ is solvable whenever it is feasible. (Feasibility means that there exists $x\geq 0$ such that $Ax+q\geq 0$.)
\end{itemize}
If $A$ has the property $(*)$, we say that $A$ belongs to the class $(*)$; we use the same bold letter to denote the property as well as the corresponding class.
\end{definition}

It is easy to see that for any permutation matrix $P$, the map $A\mapsto PAP^T$ keeps each of the above classes invariant. In particular, this holds when $P$ is the {\it permutation matrix $J$ which has all $1$s on the antidiagonal.} We remark that when $A$ is a lower triangular matrix, $JAJ$ becomes an upper triangular matrix with diagonal entries rearranged. 

The inclusion/equivalence statements in the following theorem are well-known \cite{cottle-pang-stone}.

\begin{theorem}\label{inclusion theorem}
    Consider the LCP classes defined above. Then the following hold:
    \begin{itemize}
        \item [$(i)$]
        $\bfP \subseteq \bfE\subseteq \bfR\subseteq \bfQ,\quad
        \bfP_0\cap \bfR_0\subseteq \bfE_0\cap \bfR_0= {\bfR}^*\subseteq\bfR\subseteq \bfR_0, \quad \mbox{and}\quad 
          \bfQ\subseteq \bfS$.
         \item [$(ii)$] Suppose $A\in \bfP_0$. Then $A\in \bfQ$ if and only if $A\in \bfR_0$.
        \item [$(iii)$] If $A$ is an $\bfS$-matrix, then $A$ cannot have a nonpositive row. In particular, if $A$  is a $\bfQ$-matrix, then every row of $A$ must have a positive entry.
        \end{itemize}
\end{theorem}

\begin{theorem}\label{block matrix Thm}
     Let $A \in \MnR$, $n\geq 2$, be a block matrix of the form \begin{equation}\label{block form}
         A=\begin{bmatrix}
        B&C\\D&E
         \end{bmatrix},
     \end{equation}
    where $B$ and $E$ are square matrices.
    \begin{itemize}
    \item [$(i)$] Suppose $D\geq 0$. Then $A\in \bfR_0\Rightarrow B\in \bfR_0$.
    \item [$(ii)$] Suppose $D,E\geq0$ and $d=(d_1,d_2)>0$. Then
    \begin{itemize}
        \item  [$(a)$] 
     $A\in \bfQ \Rightarrow B\in \bfQ$.
     \item [$(b)$] $B,E\in \bfR_0\Rightarrow A\in \bfR_0$.
     \item [$(c)$] $B\in \bfR(d_1),\, E\in \bfR_0\Rightarrow A\in \bfR(d)$.
     \item [$(d)$]  $B \in \bfR, \,E \in \bfR_0\Rightarrow A \in \bfR$.
     \item [$(e)$] $B\in \bfR^*,\, E \in \bfR_0\Rightarrow A\in \bfR^*.$
     \end{itemize}
     
        \end{itemize}
\end{theorem}

\begin{proof} In what follows, for an $x$ in the domain of $A$, we write $x=(x_1,x_2)$, where $x_1$ and $x_2$ are compatible with $B$ and $E$, respectively. \\
$(i)$ Suppose $D\geq 0$. If $0\leq  x_1\perp Bx_1\geq 0$, then $x=(x_1,0)$ satisfies
$0\leq x\perp Ax\geq 0$. Thus, $A\in \bfR_0\Rightarrow B\in \bfR_0$.\\
$(ii)$ Henceforth, we assume that $D,E\geq0$ and $d=(d_1,d_2)>0$.\\
$(ii)(a)$  Assume $A\in \bfQ$ and $p\in \R^k$, where $B$ has order $k$. Let $q=(p,r)$, where  $0 < r \in \R^{n-k}$. If $x=(x_1,x_2)\in \R^k\times \R^{n-k}$ is a solution of $\LCP(A,q)$, then we have
$x_2\geq 0, Dx_1+Ex_2+r>0$ and $\langle x_2,Dx_1+Ex_2+r\rangle=0.$ Consequently, 
$x_2=0.$
It follows that  $x_1$ solves $\LCP(B,p)$. This proves that $B\in \bf \bfQ$.\\
$(ii)(b)$ Assume that $B,E \in \bfR_0$. To show that $A \in \bfR_0$, consider $0\leq x\perp Ax\geq 0$. Then, $x_1\geq 0,x_2\geq 0$ and 
$$Bx_1+Cx_2\geq 0, \,Dx_1 +Ex_2 \geq 0,\,\,\mbox{with}\,\, x_1^T(Bx_1+Cx_2) +x_2^T(Dx_1 +Ex_2)=0.$$
From these, we get $x_2^T(Dx_1 +Ex_2)=0$.
Consequently, as $D,E \geq0$, we get $Ex_2 \geq 0$ with $x_2^TEx_2=0$. Since $E \in \bfR_0$,  $x_2=0$. Putting $x_2=0$ in the previous expressions, we get $Bx_1\geq 0$, $x_1^TBx_1=0$; hence $x_1=0$ (as $B \in \bfR_0)$. Therefore $A \in \bfR_0$.\\
$(ii)(c)$ Suppose that $B\in \bfR(d_1)$ and $E\in \bfR_0$. From Item $(b)$, $A \in \bfR_0$. To show that $A \in {\bfR}(d)$, let $x$ be a solution of $\LCP(A,d)$. Then $x_2 \geq 0, Dx_1+Ex_2+d_2 \geq 0$ with $x_2^T(Dx_1+Ex_2+d_2)=0$. As $Dx_1+Ex_2+d_2>0$, we get $x_2=0$. Consequently, $x_1$ solves $\LCP(B,d_1)$. As $B \in \bfR(d_1)$, we have $x_1=0$. This implies that $A \in \bfR(d)$.\\
$(ii)(d)$ This is an easy consequence of Item $(ii)(c)$.\\
$(ii)(e)$ Suppose $B\in \bfR^*, \,E \in \bfR_0$. Then $B\in \bfR_0$ and $B\in {\bfR}(d_1)$ for all $d_1>0$. Then from Item $(c)$, $A \in {\bf R}(d)$ for every $d=(d_1,d_2)>0$; hence $A \in \bfR^*$.
\end{proof}

Assume that the matrix $A$ is given in the block form (\ref{block form}). Suppose $E$ is invertible. Then, the principal pivotal transform of $A$ with respect to $E$ is defined by
    $$\widetilde{A}:=
    \begin{bmatrix} A/E & CE^{-1}\\-E^{-1}D & E^{-1}
    \end{bmatrix},$$
    where $A/E:=B-CE^{-1}D$ is the Schur complement of $A$ with respect to $E$. The following results are useful:
    \begin{itemize}
        \item  $\det A= \det (A/E)\cdot \det E$ (Schur determinantal formula).
        \item  $A\in \bfQ\Leftrightarrow \widetilde{A}\in \bfQ$ \cite{cottle-pang-stone}. 
        \item  If $E^{-1} \geq 0$ and $D \leq 0$, then $A \in \bfQ \Rightarrow A/E \in \bfQ$ (from Theorem \ref{block matrix Thm}$(ii)(b)$). 
    \end{itemize}

 We briefly recall the concept of (topological) degree of a matrix $A\in \bfR_0$. For such a matrix, we consider the nonlinear function
$$f(x):=\min\{x,Ax\}=x-(x-Ax)^+\quad (x\in \Rn),$$
where for any $z\in \Rn$, $z^+=\max\{z,0\}.$
As $f(x)=0\Rightarrow x=0$, given any bounded open set $\Omega$ containing $0$ in $\Rn$, the (topological) degree of $f$,
$\deg (f,\Omega,0)$, is well-defined as an integer \cite{ facchinei-pang}.  This is the (LCP) degree of $A$:
$$\deg A:= \deg(f,\Omega,0).$$
Usually, one computes the $\deg A$ by using various properties of degree. In some instances, $\deg A$ (for $A\in \bfR_0$)  can be computed as follows:
 Suppose for some $q$, $\LCP(A,q)$ has a finite number of nondegenerate solutions with $z$ denoting any one of them.
  (It is known that there is always such a $q$. Here, nondegeneracy means: $z+Az+q>0$.) Corresponding to $z$, let $I:=\{k:z_k>0\}$ (support of $z$), and $A_{II}$ denote the corresponding supporting matrix (which is a submatrix of $A$ corresponding to the indices in $I$). Then, see \cite{gowda-degree1993}, p. 871,
  $$\deg A:=\sum \mbox{sgn} \det A_{II},$$ where the sum is taken  over all solutions of $\LCP(A,q)$.
  We remark that $\deg A$ can also be computed by using the function $g(x)=x^+-Ax^-$, where $x^+:=\max\{x,0\}$ and $x^-:=x^+-x$, see \cite{cottle-pang-stone}, Section 6.1. For various general properties of degree, we refer to Proposition 2.1.3 in \cite{facchinei-pang}. In connection with linear complementarity problems, we have the  following:
\begin{itemize}
    \item If $\deg A\neq 0$, then $A\in \bfQ$.
    \item If $A \in \bfR$, then $\deg A=  1$.
    \item If $A$ has the block form (\ref{block form}) with $A\in \bfR_0$ and $E$ invertible, then $\widetilde{A}\in \bfR_0$ and $$\deg \widetilde{A} = \deg A\cdot \mbox{sgn}\,\det E$$ (Cf. \cite{cottle-pang-stone}, Theorem 6.6.23.)
    \end{itemize}

  In Theorem \ref{block matrix Thm}, Item $(ii)(a)$ shows that $A\in \bfQ\Rightarrow B\in \bfQ$, while Item $(ii)(c)$ provides conditions for the converse. In our next result, we consider a variation of the (converse) requirements that $B\in \bfR(d_1)$ and $E\in \bfR_0$.

\begin{proposition}\label{degree result}
Let $A$ be given in the block form (\ref{block form}). Suppose that 
    \begin{itemize}
        \item $B,E\in \bfR_0$,
        \item $D,E\geq 0.$
        \end{itemize}
     Then,  A$\in \bfR_0$ and  
     $\deg A=\deg B\cdot \,\deg E.$
     In particular, 
     \begin{itemize}
         \item [$(i)$] When $E$ is of size $1\times 1$ (with a positive entry), $\deg A=\deg B$;
         \item [$(ii)$] When $\deg B$ and $\deg E$ are nonzero, $A\in \bfQ$.
         \end{itemize}
    \end{proposition}

    \begin{proof}
    That $A\in \bfR_0$ follows from Theorem \ref{block matrix Thm}$(ii)(b)$. More generally, $A(\varepsilon)\in \bfR_0$ for all $0\leq \varepsilon \leq 1$, where
    $$A(\varepsilon):=\begin{bmatrix}
        B & \varepsilon C\\
        \varepsilon D& E
        \end{bmatrix}.$$
        By homotopy invariance of degree,  see \cite{facchinei-pang}, Section 2.1.1,
        $$\deg A(1)=\deg A(0).$$
        As $\deg A(0)=\deg B\cdot \,\deg E$, see \cite{facchinei-pang}, Proposition 2.1.3$(h)$,
        we have $\deg A=\deg B\,\cdot\,\deg E$.
        The additional Item $(i)$ follows easily; the second statement $(ii)$ follows from $\deg A\neq 0\Rightarrow A\in \bfQ$.   
    \end{proof}

%%%%%%%%%%%%%%%%%%%%%%%%%%%%%%%%
\section{Triangular matrices}
A matrix $A=[a_{ij}] \in \MnR$ is said to be {\it upper triangular} if $a_{ij}=0$ for all $i>j$; it is said to be lower triangular if $A^T$ is upper triangular, and triangular if $A$ or $A^T$ is upper triangular.

\begin{theorem}\label{triangular}
    Let $A=[a_{ij}] \in \MnR$ be a triangular matrix.  Then the following are equivalent:
\begin{itemize}
\item [$(i)$] $A\in \bfQ$.
\item [$(ii)$] The diagonal of $A$ is positive.
\item [$(iii)$] $A\in \bfP$.
\item [$(iv)$] $A\in \bfE.$
\item [$(v)$] The diagonal of $A$ is nonnegative and $A\in \bfR_0$.
\item [$(vi)$]  $A\in \bfR^*.$
\end{itemize}
\end{theorem}

\begin{proof} We prove the above equivalence {\it  assuming that $A$ is an upper triangular matrix}; the lower triangular case is seen by considering the matrix $JAJ$ (where $J$ is the permutation matrix with all $1$s on the antidiagonal).\\
    $(i) \Rightarrow (ii)$: 
    Let $A\in \MnR$ and a $\bfQ$-matrix. We establish $(ii)$ by inducting on $n$. When $n=1$, $a_{11}>0$ as $A$ is a $\bfQ$-matrix. Assume $n\geq 2$. By considering a solution $x$ of $\LCP(A,-\bfone)$, we see that $a_{nn}x_n-1=(Ax)_n-1\geq 0$; hence, $a_{nn}>0$. So the last row in our upper triangular matrix $A$ is nonnegative.  Let $B$ be the submatrix of $A$  obtained by removing the last row and last column of $A$. By Item $(ii)(a)$ in Theorem \ref{block matrix Thm}, $B\in \bfQ$. 
    As $B$ is upper triangular, by our induction hypothesis, $B$ has positive diagonal. As we have already shown that $a_{nn}>0$, we see that the diagonal of $A$ is positive. \\
     $(ii) \Rightarrow (iii)$: This holds as any principal minor of an upper triangular matrix is the product of some diagonal entries. \\
     $(iii) \Rightarrow (iv)$: This holds as $\bfP\subseteq \bfE$.
     \\
     $(iv)\Rightarrow (v)$: This holds as  $\bfE\subseteq \bfR_0$ and the diagonal of an $\bfE$-matrix is positive.\\
     $(v) \Rightarrow (vi)$: Suppose $(v)$ holds. As $A$ is upper triangular, $A$ is now a $\bfP_0$-matrix. As $A\in \bfR_0$, we see from Theorem \ref{inclusion theorem}, $A\in \bfR^*$.\\
      $(vi) \Rightarrow (i)$: This follows from Theorem \ref{inclusion theorem}.
    \end{proof}

\begin{example}\label{negative diagonal example}
   Consider the following upper triangular matrices, each a non $\bfQ$-matrix. While the matrix $A_1$ shows that nonnegativity is needed in Item $(v)$ of the theorem, $A_2$ shows that the $\bfS$-property cannot replace the $\bfR_0$-condition (so cannot be part of the above theorem).
$$A_1=\begin{bmatrix}
        -1 & ~~0 \\ ~~0 & -1
    \end{bmatrix}\quad \mbox{and}\quad A_2=\begin{bmatrix}
        0 &1 \\ 0&1
    \end{bmatrix}.$$
    \end{example}

We now consider a variant of the above theorem. To an upper triangular matrix, we adjoin a nonnegative row.

\begin{theorem}\label{Upper triangular+ sw corner}
   Let $A=[a_{ij}] \in \MnR$, $n\geq 2$, be a matrix of the following block form:
    \begin{equation}\label{extension of triangular}A=\begin{bmatrix}
        B &c\\
        d^T& a_{nn}
    \end{bmatrix},\end{equation}
     where $B$ is an upper triangular matrix of order $n-1$, 
     \begin{center}
         $0\leq d \in \R^{n-1}$, and  $a_{nn} >0$.
         \end{center} 
         Then the following statements are equivalent:
    \begin{enumerate}
        \item[$(i)$] $A\in \bfQ$.
        \item [$(ii)$] $A\in \bfE$.
        \item[$(iii)$] The diagonal of $A$ is positive.
         \item[$(iv)$] The diagonal of $A$ is nonnegative and $A\in \bfR_0$.
         \item[$(v)$]  $A\in \bfR^*$.
    \end{enumerate}
   
\end{theorem}

\begin{proof}
          $(i) \Rightarrow (ii)$: Let $A\in \bfQ$. As the last row of $A$ is nonnegative, by Theorem \ref{block matrix Thm}(ii)(a), $B$ is a $\bfQ$-matrix. Then by Theorem \ref{triangular}, $B\in \bfE$. Since the last row of $A$ is nonnegative with $a_{nn}>0$, we see that $A\in \bfE$.\\
    $(ii) \Rightarrow (iii)$: This is obvious as the diagonal of an $\bfE$-matrix is positive. \\
    $(iii) \Rightarrow (iv)$: Assume that $(iii)$ holds. Then the first part of $(iv)$ is immediate. Now, since $A$ has positive diagonal, so does $B$. Therefore, from Theorem \ref{triangular}, $B \in \bfR_0$. As $a_{nn} >0$, i.e., $[a_{nn}] \in \bfR_0$, from Theorem \ref{block matrix Thm}$(ii)(b)$, $A \in \bfR_0$.\\
    $(iv)\Rightarrow (v)$:  Suppose $A$ has nonnegative diagonal and $A \in \bfR_0$. Then $B$ has nonnegative diagonal entries; moreover, as $d\geq 0$, from Theorem \ref{block matrix Thm}$(i)$, $B \in  \bfR_0$. Then $B \in \bf R^*$ from Theorem \ref{triangular}. As $a_{nn}>0$, from Theorem \ref{block matrix Thm}$(ii)(e)$,  $A \in \bf R^*$. \\
    $(v) \Rightarrow (i)$: This follows from Theorem \ref{inclusion theorem}.
\end{proof}

%%%%%%%%%%%%%%%%%%%%%%%%%%%%%%%%%%%%
\section{Bidiagonal southwest matrices}
A square matrix is said to be {\it bidiagonal} \cite{higham} if  nonzero entries appear only on the main diagonal and superdiagonal (or subdiagonal). With a slight modification, we define a {\it bidiagonal southwest} matrix as follows:

\begin{definition} 
$A=[a_{ij}] \in \MnR$, $n \geq 2$, is said to be a bidiagonal southwest (bdsw) matrix if $A$ has the following form:
$$A = \begin{bmatrix}
    a_{11} & a_{12} &  & 0 &  &  0\\
     & a_{22} & a_{23} &  &  &  \\
     &  & a_{33} & a_{34} &   &  \\
    0 &    &   & \ddots & \ddots & 0\\
     &  &  &  & a_{(n-1)(n-1)} & a_{(n-1)n} \\
    a_{n1} &  & 0 &  & 0 & a_{nn}
\end{bmatrix}.$$

\end{definition}
In the above, an entry $a_{ij}$ can be zero. In particular, when $a_{n1}=0$, the above matrix reduces to a {\it bidiagonal matrix.} We note that in the recent paper \cite{pratihar-sivakumar2024}, where the concept of `bdsw matrix' was introduced, all the above-mentioned entries $a_{ij}$ are assumed to be nonzero. Under this assumption, the directed graph corresponding to a bdsw matrix is a simple directed cycle with loops, and the inverse of a nonsingular bdsw matrix is full (that is, all entries are nonzero), see the cited paper for further details.
\\

While dealing with bdsw matrices, for ease of presentation, we use the following terminology:
{\it  In the $k$th row of a bdsw matrix, we say that $a_{kk}$ is the diagonal entry; $a_{k(k+1)}$ $($or $a_{n1}$ when $k=n)$ is the (relevant) `off-diagonal' entry.}

 \begin{remark}   An easy computation shows that for the above bdsw matrix,
    \begin{equation}\label{determinant of bdsw matrix}
    \det A= a_{11}a_{22}\cdots\, a_{nn}+(-1)^{n+1}a_{12}a_{23}\cdots a_{(n-1)n}a_{n1}.
    \end{equation}
\end{remark}

 Generally, for a bdsw matrix $A$ and a permutation matrix $P$ $(\neq I)$, the matrix $PAP^T$ need not be a bdsw matrix. However, as we see below, in at least one special case, we retain the bdsw structure.  We omit the routine verification of the following result.
\begin{proposition}\label{special permutation invariance of bdsw property}
Suppose $A=[a_{ij}] \in \MnR$ is a bdsw matrix. Let $1\leq k< n$. With $\{e_1,e_2,\ldots, e_n\}$ denoting the standard coordinate system (of column vectors) in $\Rn$, consider the permutation matrix 
    \begin{equation}\label{Permutation matrix}
        P:=[e_{n-k+1}, e_{n-k+2}, \dotsc, e_n,e_1, \dotsc, e_{n-k-1}, e_{n-k}].
    \end{equation} 
    Let $B:=PAP^T$. Then,
    \begin{itemize}
        \item [$(1)$]
    $B$ is a bdsw matrix,
    \item [$(2)$] $\{b_{11},b_{22},\ldots, b_{nn}\}=\{a_{11},a_{22},\ldots, a_{nn}\}$.
    \item [$(3)$]   $b_{n1}=a_{k(k+1)}, \,b_{nn}=a_{kk}\quad \mbox{and}\quad  b_{kk}=a_{nn},\, b_{k(k+1)}=a_{n1}.$
    \end{itemize}
    \end{proposition}
    
In what follows, we consider four types of bdsw matrices; {\it we exclude bdsw matrices with nonpositive rows} (as they cannot have the $\bfQ$-property, see Theorem \ref{inclusion theorem}$(iii)$).

%%%%%%%%%%%%%%%%%%%%%%%%%%%%%%%%%
\section{Type-I bdsw matrices}
{\it Type-I bdsw matrices are bdsw matrices that contain at least one (nonzero) nonnegative row.} 
For the $\bfQ$-property of Type-I bdsw matrices, we formulate our results based on the signs of $a_{n1}$ and $a_{nn}$.

\begin{itemize}
\item Case 1: $a_{n1}\geq 0, \,a_{nn}>0$. (See Theorem \ref{Bdsw, an1>=0, ann>0}.)
% \item  Case 2: $a_{n1}>0, \,\,a_{nn}>0.$ (See Theorem \ref{Bdsw, an1>0, ann>0}.) \TR{We can combine Case 1 and Case 2 and use the Theorem \ref{Upper triangular+ sw corner}.} \TB{Yes, go ahead and make changes throughout the paper.}
    \item Case 2: $a_{n1}> 0,\,a_{nn}= 0$. (See Theorem \ref{Bdsw, an1>0, ann=0}.)
    \item Case 3: $a_{n1}<0,\,a_{nn}>0$. (See Theorem \ref{bdsw with kth row nonnegative}.)
     \item Case 4: $a_{n1}> 0,\,a_{nn}< 0$. ($A \not\in \bfQ$, see Theorem \ref{bdsw a_n1>0,ann<0}.)
   \end{itemize}

The following result is a special case of Theorem \ref{Upper triangular+ sw corner}.
\begin{theorem}\label{Bdsw, an1>=0, ann>0}
     Let $A=[a_{ij}]\in \MnR$, $n\geq 2$, be a bdsw matrix with 
   $$a_{n1} \geq 0,\,a_{nn}> 0.$$ 
     Then, the equivalent statements of Theorem \ref{Upper triangular+ sw corner}
hold.
\end{theorem}

% The following result comes from Theorem  \ref{Upper triangular+ sw corner}. 

% \begin{theorem}\label{Bdsw, an1>0, ann>0}
%      Let $A=[a_{ij}]\in \MnR$, $n\geq 2$, be a bdsw matrix with 
%    $$a_{n1}> 0,\,a_{nn}> 0.$$ 
%      Then, the equivalent statements of Theorem \ref{Upper triangular+ sw corner}
% hold.
    
%     \end{theorem}

In our next result, we consider Case 2. We observe that the condition $A\in \bfE$ is no longer part of the equivalence.
\begin{theorem}\label{Bdsw, an1>0, ann=0}
     Let $A=[a_{ij}]\in \MnR$, $n\geq 2$, be a bdsw matrix with 
   $$a_{n1}> 0,\,a_{nn}= 0.$$ 
     Then, the following are equivalent:
     \begin{itemize}
          \item [$(i)$] $A\in \bfQ$.
        \item [$(ii)$] $a_{ii}>0$ for $i=1,2,\dotsc,n-1$, and $A$ has negative superdiagonal.
          \item [$(iii)$] $A\in \bfR^*$.
          \end{itemize}
    
Furthermore, in each case, $\det A>0$.
\end{theorem}
\begin{proof}
      It is given that $A$ is a bdsw matrix. Let $B$ be the leading principal submatrix of $A$ of order $n-1$. Note that $B$ is a bidiagonal matrix.\\
    $(i) \Rightarrow (ii)$: Let $A\in \bfQ$. As the last row of $A$ is nonnegative, by Theorem \ref{block matrix Thm}$(ii)(a)$, $B$ is a $\bfQ$-matrix. Then by Theorem \ref{triangular}, $B$ has a positive diagonal. Therefore $a_{ii}>0$ for all $i=1,2,\dotsc,n-1$. We now show that $A$ has negative superdiagonal by proving the following:\\
    {\it Claim:} $a_{i(i+1)} <0$ for all $1\leq i \leq n-1$.\\
    Assuming the contrary, consider the least index $k$, $1 \leq k \leq n-1$, such that $a_{k(k+1)} \geq 0$. In $\Rn$, consider  $$q=(1,1, \dotsc, 1, -1)^T$$
    and let $x=(x_1,x_2,\ldots, x_n)^T$ be a solution of $\LCP(A,q)$.\\
    We first show that $k>1$. Assuming the contrary, suppose $k=1$. In this case, $a_{11}>0$ and $a_{12}\geq 0.$ Then, the complementarity condition $x_1(a_{11}x_1+a_{12}x_2+1)=0$ implies that $x_1=0$. But then, \begin{equation}\label{x1 cannot be zero}
    0\leq (Ax+q)_n=a_{n1}x_1+a_{nn}x_n -1=a_{n1}0+0x_n-1=-1
    \end{equation}
    yields a contradiction. Hence, $k>1$.
    Now, as $a_{kk}>0,\, a_{k(k+1)} \geq 0$, and $q_k=1$, from   the complementarity condition
    $x_k(Ax+q)_k=0$, we get $x_k=0$. Then, 
    $$0=x_{k-1}(Ax+q)_{k-1}=x_{k-1}(a_{(k-1)(k-1)}x_{k-1}+a_{(k-1)k}x_k+1)$$ implies, as the diagonal of $B$ is positive, that $x_{k-1}=0$. (Note that the coefficient $a_{(k-1)k}$ plays no role here.) Repeating this (backward substitution) argument, we see that $x_{k-2}=0$, etc., eventually getting $x_1=0$. But then, (\ref{x1 cannot be zero}) yields a contradiction. This proves our claim. Hence, all superdiagonal elements are negative. \\
    $(ii) \Rightarrow (iii)$: Suppose $(ii)$ holds. First, we show that $A\in \bfR_0$. Consider $0\leq x\perp Ax\geq 0$, let $x=(\tilde{x},x_n)$, where $x_n\in \R$. Suppose, if possible, $x_n\neq 0$ so $x_n > 0$. Then the last complementarity condition $x_n(a_{n1}x_1+a_{nn}x_n)=0$ together with $a_{n1}>0$ and $a_{nn}=0$, yields $x_1=0$. Now, the first condition $(Ax)_1\geq 0$, in conjunction with $a_{12}<0$, implies $x_2=0$. This, in turn, with $a_{23}<0$ and $(Ax)_2\geq 0$ implies $x_3=0$. A continuation of this argument eventually leads to $x_{n-1}=0$ and, via $(Ax)_{n-1}\geq 0$, to $x_n=0$. This contradicts our assumption that $x_n\neq 0$. Therefore $x_n=0$, and $\tilde{x}$ is a solution of $\LCP(B,0)$; as $B$ is bidiagonal with positive diagonal, by Theorem \ref{Upper triangular+ sw corner}, $B\in \bfR_0$; hence $\tilde{x}=0$. Thus, $A \in \bf R_0$.\\
 Now, we show that $A\in \bfR(d)$ for any $d>0$. Let $d>0$ and $y$ be a solution of $\LCP(A,d)$. Let $y=(\widetilde{y},y_n)$, where $y_n\in \R$. As $d> 0$, $a_{n1}> 0$, and $a_{nn}=0$, we see, by complementarity, $y_n=0$. But then, $\widetilde{y}$ becomes a solution of $\LCP(B,\widetilde{d})$, where $d=(\widetilde{d},d_n)$. As $B$ (which is a bidiagonal matrix) has a positive diagonal, by Theorem \ref{Upper triangular+ sw corner},  $B\in \bfR(\widetilde{d})$. Thus, $\widetilde{y}=0$. We see that $A\in \bfR(d)$ for any $d>0$ and hence $A \in \bfR^*$.\\
 $(iii) \Rightarrow (i)$: This follows from Theorem \ref{inclusion theorem}.\\
 Now suppose that one of the equivalent conditions, say, Item $(ii)$, holds. Then $A$ has a negative superdiagonal. As $A$ is a bdsw matrix with $a_{nn}=0$, from (\ref{determinant of bdsw matrix}),
    $$\det A= (-1)^{n+1}a_{n1} \prod\limits_{i=1}^{n-1}a_{i(i+1)}.$$ As $a_{n1}>0$ and $a_{i(i+1)}<0$ for all $i \leq n-1$, we have $\det A >0$.
 \end{proof} 

In Theorem \ref{Bdsw, an1>=0, ann>0}, the condition $A\in \bfE$ is part of the equivalence. However, in Theorem \ref{Bdsw, an1>0, ann=0} this is no longer the case. 
\begin{example}
$$    A=\begin{bmatrix}
        1 & -1\\1 & ~~0
        \end{bmatrix}
$$
is a bdsw matrix satisfying condition $(ii)$ of the above theorem, but $A\not \in \bfE$.
\end{example}

The following example shows that in the above theorem, the determinant being positive is only a necessary condition, but not sufficient.
\begin{example} The bdsw matrix $$A=\begin{bmatrix}
    1&~~1&~~0&~~0\\
    0&~~1&~~1&~~0\\
    0&~~0&~~1&-1\\
    1&~~0&~~0&~~0
\end{bmatrix}$$ has positive determinant.
It is easy to  verify that $\LCP(A,q)$ has no solution for $q=(0,0,0,-1)^T$. Thus, $A \notin \bfQ$.

\end{example}

Motivated by the above theorem, we ask if a similar result holds for an upper triangular matrix of the form (\ref{extension of triangular}), where $a_{nn}=0$ and $a_{n1}>0$. The following example answers this in the negative.

\begin{example}\label{Counter for extension of bdsw}
    Consider the following matrix $$A=\begin{bmatrix}
        1&-1&~~1\\
        0&~~1&-1\\
        1&~~0&~~0
    \end{bmatrix}.$$ Here $A$ has negative superdiagonal with $a_{11},a_{22}>0$, but $\LCP(A,q)$ has no solution for $q=(0,0,-1)^T$. Therefore $A \notin \bfQ$. 
\end{example}

We now consider Case 3.

\begin{theorem}\label{bdsw with kth row nonnegative}
    Suppose $A=[a_{ij}]\in \MnR$, $n\geq 2$, is a Type-I bdsw matrix with 
   $$a_{n1}<0,\,a_{nn}> 0.$$
   Let the $k$th row of $A$, $k\neq n$,  be nonnegative. Then the following statements are equivalent:
     \begin{itemize}
          \item [$(i)$] $A\in \bfQ$.
        \item [$(ii)$] One of the following holds: 
        \begin{enumerate}
            \item[$(a)$] The diagonal of $A$ is positive.
            \item[$(b)$] $a_{kk}=0,\,a_{k(k+1)}>0$ and $a_{ii}>0,\, a_{i(i+1)}<0$ for all $i \neq k$. 
        \end{enumerate}
            \item [$(iii)$] $A\in\bfR^*$.
          \end{itemize}
\end{theorem} 

\begin{proof}
    We fix a $k$ such that  $a_{kk} \geq 0,\, a_{k(k+1)} \geq 0$. Consider the permutation matrix $P$ as in (\ref{Permutation matrix}). Then $B=[b_{ij}]:=PAP^T$ is also a bdsw matrix with the same diagonals as those of $A$ (except permutated) and $b_{n1}=a_{k(k+1)}\geq 0,\, b_{nn}=a_{kk} \geq 0.$ \\
     $(i) \Rightarrow (ii)$: Assume that $A \in \bfQ$. Then $B\in \bfQ$; hence, at least one of $b_{n1},\,b_{nn}$ is positive. We consider two cases.\\
     {\it Case $1:$} $b_{n1}\geq 0,\,b_{nn}>0$.\\ Since $B$ is a bdsw matrix of the form (\ref{extension of triangular}), from Theorem \ref{Bdsw, an1>=0, ann>0}, we can say $B$ has  positive diagonal entries; consequently, $A$ has positive diagonal entries.\\
     {\it Case $2:$} $b_{n1}>0,\,b_{nn}=0$. \\In this case, $a_{k(k+1)}>0,\, a_{kk}=0$. Furthermore, from Theorem \ref{Bdsw, an1>0, ann=0}, we can say that $b_{ii}>0$ for all $i \leq n-1$ and $B$ has negative superdiagonal. Therefore, $a_{ii} >0$ and $a_{i(i+1)} <0$ for all $i \neq k$.\\
     $(ii) \Rightarrow (iii)$: Suppose that Item $(a)$ of $(ii)$ holds, i.e., $A$ has a positive diagonal. Then, $B$ has a positive diagonal. As $b_{nn}>0,\,b_{n1}\geq 0$, from Theorem \ref {Bdsw, an1>=0, ann>0} we can say that $B \in \bf R^*$ and hence, $A\in \bfR^*$.\\
     Suppose Item $(b)$ of $(ii)$ holds. Here $b_{nn}=0,\, b_{n1}>0$. Then using Theorem \ref{Bdsw, an1>0, ann=0} we can say that $B\in \bfR^*$; consequently, $A \in \bfR^*$.\\
     $(iii) \Rightarrow (i)$: This follows from Theorem 2.1.  
\end{proof}

We finally consider Case 4, with a negative conclusion.

\begin{theorem}\label{bdsw a_n1>0,ann<0}
    Suppose $A=[a_{ij}]\in \MnR$, $n\geq 2$, be a Type-I bdsw matrix with 
   $$a_{n1}>0,\,a_{nn}< 0.$$
   Then, $A\not\in \bfQ$.
\end{theorem}

\begin{proof}
    As $A$ is a Type-I bdsw matrix, we assume that the $k$th row of $A$, $k\neq n$, is nonnegative. Consider $B=[b_{ij}]:=PAP^T$, where $P$ is the permutation matrix specified in (\ref{Permutation matrix}). Then $B$ is a bdsw matrix with last row nonnegative and  $b_{(n-k)(n-k)}=a_{nn}<0$. Suppose $A\in \bfQ$, in which case, $B\in \bfQ$. Then, depending on the entries of the last row of $B$, we can apply either Theorem \ref{Bdsw, an1>=0, ann>0} or Theorem \ref{Bdsw, an1>0, ann=0} to conclude that the first $n-1$ diagonal entries of $B$ are positive. However, this is not possible as $b_{(n-k)(n-k)}=a_{nn}<0$. This contradiction shows that $A\not \in \bfQ$.
\end{proof}

By considering various cases of a Type-I bdsw matrix, we arrive at the following necessary condition.
\begin{corollary}
    Suppose $A$ is a Type-I bdsw matrix. If $A\in \bfQ$, then the diagonal of $A$ is nonnegative.
\end{corollary}
%%%%%%%%%%%%%%%%%%%%%%%%%%%%%%%
\section{Type-II bdsw matrices}
{\it A Type-II bdsw matrix is a bdsw matrix, where in each row, the diagonal entry is positive, and the (relevant) off-diagonal entry is negative.} Such a matrix is a $\bfZ$-matrix. So,  a Type-II bdsw matrix has the  $\bfQ$-property if and only if it is a $\bfP$-matrix (equivalently, an $\bfS$-matrix). Since all the proper principal minors of a Type-II bdsw matrix are positive, we see that a Type-II bdsw matrix $A$ is in $\bfQ$ if and only if $\det A>0$. For the record, we state this as follows:

\begin{theorem} \label{type2 bdsw matrix theorem}
A Type-II bdsw matrix is in $\bfQ$ (equivalently, in $\bf P$, or in $\bfR^*$) if and only if its determinant is positive.
\end{theorem}
%%%%%%%%%%%%%%%%%%%%%%%%%%
\section{Type-III bdsw matrices}
{\it A Type-III bdsw matrix is a bdsw matrix where in each row, the diagonal entry is negative, and the (relevant) off-diagonal entry is positive.} We observe that such a matrix is the negative of a Type-II bdsw matrix. For convenience, we write such a matrix in the following form:

\begin{equation}\label{type3 bdsw matrix}
A = \begin{bmatrix} 
    -a_{11} & a_{12} &  & 0 &  & 0 \\
     & -a_{22} & a_{23} &  &  &  \\
     &  & -a_{33} & a_{34} &   &  \\
    0 &    &   & \ddots & \ddots & 0\\
     &  &  &  & -a_{(n-1)(n-1)} & a_{(n-1)n} \\
    a_{n1} &  & 0 &  & 0 & -a_{nn}
\end{bmatrix},
\end{equation}
where every $a_{ij}>0$.

Our characterization result is the following:
\begin{theorem}\label{type3 theorem} 
Let $A\in \MnR$ be a Type-III bdsw matrix as in (\ref{type3 bdsw matrix}), $n\geq 2$. Then the following are equivalent:
\begin{itemize}
\item [$(i)$] $A\in \bfQ$.
\item [$(ii)$] $(-1)^{n+1}\det A>0$.
\item [$(iii)$] $A \in \bfR_0$ and $\deg A =\pm 1.$
\end{itemize}
\end{theorem}

First, we prove the following lemma.
\begin{lemma}\label{lemma1}
Let $A$ be as in (\ref{type3 bdsw matrix}), $n\geq 2$. Then, the following statements hold:
\begin{itemize}
\item [$(i)$]  $\det A=0$ if and only if there there is a positive vector $u$  such that $A^Tu=0.$
    \item [$(ii)$] If $\det A\neq 0$, then
    $$A^{-1}=\frac{(-1)^{n+1}}{\det A} B,\,\mbox{where}\, B>0.$$
    \item [$(iii)$] If $\det A\neq 0$, then $A\in \bfR_0$.
    \item [$(iv)$] If $(-1)^{n+1}\det A>0$, then    $\deg A =\pm 1.$
\end{itemize}
\end{lemma}

\begin{proof}
From (\ref{determinant of bdsw matrix}), we see that 
$$\det A=(-1)^n\big [ a_{11}a_{22}\cdots a_{nn}\,-\, a_{12}a_{23}\cdots a_{(n-1)n}a_{n1}\big ].$$
$(i)$ We prove the `only if' part. When  $\det A=0$, we have $\det A^T=0$; so, there exists $ y\neq 0$ such that $A^Ty=0$. Writing $A^{T}y=0$ as a system of equations, we see that each coordinate of $y$ is a positive multiple of the first coordinate $y_1$. Taking $y_1=1$, we find a vector $u>0$ such that $A^Tu=0$.\\
$(ii)$ Suppose $\det A\neq 0$. Then the $i$th column of $A^{-1}$ is obtained by solving the equation $Ax=e_i$. For simplicity, we take $i=1$ and verify that the vector $x$ in $Ax=e_1$ is of the form 
$$x=\frac{(-1)^{n+1}}{\det A} b_1,$$
where $b_1>0$. Now, writing $Ax=e_1$ as a system of equations and using backward substitution, we see that every component of $x$ is a positive multiple of $x_1$, with $x_1=(-1)^{n+1} \frac{a_{22} \dotsc a_{nn}}{\det A}$. The stated expression for $x$ follows. The other columns of $A^{-1}$ will have similar representations. Consequently, $(ii)$ follows. \\
$(iii)$ Suppose $\det A\neq 0$. By $(i)$, either $A^{-1}>0$ or $A^{-1}<0$. In both cases, $A^{-1}\in \bfR_0$, hence $A\in \bfR_0$.\\
$(iv)$ Suppose $(-1)^{n+1}\det A>0$. By Items $(iii)$ and $(ii)$, $A\in \bfR_0$ and $A^{-1}$ is a positive multiple of $B$, where $B>0$. So, $A^{-1}\in\bfR$. Consequently, $\deg (A^{-1})=1$. Hence, by \cite{cottle-pang-stone}, Theorem 6.6.23,
$$\deg A = (\mbox{sgn}\det A) \deg A^{-1}=\pm 1.$$
\end{proof}

We now prove the Theorem \ref{type3 theorem}.
\begin{proof}
$(i) \Rightarrow (ii)$: Suppose $A\in \bfQ$. Then, $A\in \bfS$; hence there exists a vector $v>0$ such that $Av>0$. If $\det A=0$, then by the above lemma, there exists $u>0$ such that $A^Tu=0$. So, 
$$0=\langle A^Tu,v\rangle=\langle u,Av\rangle >0.$$
As this cannot happen, we must have either $(-1)^{n+1}\det A<0$ or $(-1)^{n+1}\det A>0$. In the first case, by Item $(ii)$ in Lemma \ref{lemma1},
$A^{-1}$ is a matrix with all entries negative. This cannot happen as $A\in \bfQ$. Therefore $(-1)^{n+1}\det A>0$.\\
$(ii) \Rightarrow (iii)$: Assume that $(-1)^{n+1}\det A>0$. Then by Items $(iii)$ and $(iv)$ in Lemma \ref{lemma1}, $A \in \bfR_0$ and $\deg A =\pm 1$.\\
$(iii) \Rightarrow (i)$: Suppose $A$ is a Type-III bdsw matrix of the form (\ref{type3 bdsw matrix}) with $(-1)^{n+1}\det A>0$. Then, by the above lemma, $A^{-1}$ is a positive matrix, hence a $\bfQ$-matrix. It follows that $A\in \bfQ$. 
\end{proof}

While dealing with Type-I and Type-II bdsw matrices, we showed that the $\bfQ$-property is equivalent to the $\bfR^*$-property. Such a property may not be true in the case of Type-III bdsw matrices. Here is an example. The matrix $$A=\begin{bmatrix}
    -1&~~2\\
    ~~1&-1
\end{bmatrix}$$ is in $\bfQ$ (by the above theorem), but not in  $\bfR(d)$ for any $d=(d_1,d_2)^T>0$, as $(d_1,0)^T$ and $(0,d_2)^T$ are two nonzero solutions of $\LCP(A,d)$. 

%%%%%%%%%%%%%%%%%%%%%%%%%%%%%
\section{Type-IV bdsw matrices}
{\it A Type-IV bdsw matrix is a bdsw matrix which contains no nonnegative/nonpositive row, has at least one row with a positive diagonal entry, and has at least one row with a negative diagonal entry.} So, in any row of such a matrix, each diagonal entry and its corresponding (relevant) off-diagonal entry have opposite nonzero signs.

A simple example of Type-IV bdsw matrix is 
$$A=\left [\begin{array}{rrr}
    -1 & ~~1 & ~~0 \\
         ~~0 & -1 & ~~1\\
-2 & ~~0 & ~~1
    \end{array}\right ].$$
    
\begin{theorem}\label{type4 theorem}
    Let $A\in \MnR$, $n\geq 2$, be a Type-IV bdsw matrix. Let $k$, $1\leq k<n$, be the number of negative diagonal entries in $A$. Then the following are equivalent:
    \begin{itemize}
    \item [$(i)$]  $A\in \bfQ$.
    \item [$(ii)$] $(-1)^{k+1}\det A >0$.
    \item [$(iii)$] $A\in \bfR_0$ and $\deg A=\pm 1.$
\end{itemize}
\end{theorem}
    
\begin{proof} Let $m:=n-k$ denote the number of positive diagonal entries in $A$. We prove the result by induction on $m$.
We  first consider \\ 
{\it Case $m=1:$}\\ In this case, $k=n-1$; So $A$ has exactly one row where the diagonal entry is positive with (relevant) off-diagonal entry negative. By using a suitable permutation, see (\ref{Permutation matrix}), we can assume that this is the last row of $A$. (Note that this will not affect the determinant, the $\bfR_0$ and $\bfQ$-properties; the degree is also not affected, see Theorem 2.9.8, \cite{cottle-pang-stone}).) We write
$$A=\begin{bmatrix}
    B & c\\d & a_{nn}
\end{bmatrix},$$
where $d= \big (a_{n1},0,\ldots, 0\big )$ and $c=\big (0,0,\ldots, a_{(n-1)n}\big )^T$.  Observe that $a_{n1}<0, \,a_{nn}>0$ and, as $m=1$, $a_{(n-1)(n-1)}<0,\,a_{(n-1)n}>0$.
Now consider the principal pivotal transform (PPT) $\widetilde{A}$ of $A$ obtained by pivoting on the $1\times 1$ matrix $[a_{nn}]$.  We display $\widetilde{A}$ as follows:
\begin{equation}\label{PPT}
\widetilde{A}=\begin{bmatrix}
    \widetilde{B} & \widetilde{c}\\
    \widetilde{d} & \frac{1}{a_{nn}}
    \end{bmatrix}.
    \end{equation}
    Here, $\widetilde{d}= (-\frac{a_{n1}}{a_{nn}},0,\ldots, 0 )\geq 0$ and $\widetilde{c}=(0,0,\ldots, \frac{a_{(n-1)n}}{a_{nn}})^T$. Moreover, 
    $$\widetilde{B}:=B-\frac{1}{a_{nn}} c\,d^T$$ is the Schur complement of $[a_{nn}]$ in $A$. Since $a_{nn}>0$, $a_{n1}<0$ and $a_{(n-1)n}>0$, $\widetilde{B}$ is a Type-III bdsw matrix of order $n-1$.  We observe:
    \begin{itemize}
        \item $A\in \bfR_0\Leftrightarrow \widetilde{A}\in \bfR_0\Leftrightarrow \widetilde{B}\in \bfR_0$ (see Theorem \ref{inclusion theorem} for the last equivalence),
        \item $\det A=(\det \widetilde{B})\, a_{nn}$ (from the Schur determinantal formula),
        \item When $A\in \bfR_0$, $\deg \widetilde{A}=\deg \widetilde{B}$ (from Proposition \ref{degree result}).
    \end{itemize}
    We now prove the equivalence $(i)\Leftrightarrow (ii)\Leftrightarrow (iii)$. \\
    First, suppose $(i)$ holds, so that $A\in \bfQ$.
Then, $\widetilde{A}\in \bfQ$. Since the last row of $\widetilde{A}$ is nonnegative with $a_{nn}>0$, by Theorem \ref{block matrix Thm}(ii)(a), $\widetilde{B}\in \bfQ$. By Theorem \ref{type3 theorem}, 
we must have $(-1)^{(n-1)+1}\det \widetilde{B}>0$. By above, $\det A=(\det \widetilde{B})\, a_{nn};$
as $a_{nn}>0$, with $k=n-1$, we have $(-1)^{k+1}\det A>0$.  This is Item $(ii)$. \\
Now suppose $(ii)$ holds. Then, as above, $(-1)^{(n-1)+1}\det \widetilde{B}>0$. By applying Lemma \ref{lemma1}(iv) to the Type-III bdsw matrix $\widetilde{B}$, we have $\widetilde{B}\in \bfR_0$ and 
$\deg \widetilde{B}=\pm 1$. Then, by Proposition \ref{degree result} applied to $\widetilde{A}$ with $E=[\frac{1}{a_{nn}}],$
$$\deg \widetilde{A}= \deg \widetilde{B} =\pm 1.$$ By Theorem 6.6.23 in \cite{cottle-pang-stone}, we have
$\deg A= \deg \widetilde{A} =\pm 1.$ Thus we have $(iii)$. Finally, when $(iii)$ holds, $\deg A$ is nonzero; hence, $A\in \bfQ$. 
 This completes the proof when $m=1$.
\\
{\it Case $m>1:$}\\ Assume, to apply the principle of mathematical induction, that the equivalence of $(i),(ii)$ and $(iii)$ holds for $m-1$. 
By using Proposition \ref{special permutation invariance of bdsw property} (if necessary), we assume that the last row of $A$ has a positive diagonal entry. Then $a_{nn}>0$ and $a_{n1}<0$. Now, consider the principal pivot transform, $\widetilde{A}$, of $A$ obtained by pivoting on $a_{nn}$, see (\ref{PPT}).  Consider the block matrix $\widetilde{B}$ of order $n-1$ in the northwest corner of $\widetilde{A}$. A careful observation shows that $\widetilde{B}$ is a bdsw matrix with southwest corner entry $-\frac{a_{n1}a_{(n-1)n}}{a_{nn}}$.  Since $a_{n1}<0$ and $a_{nn}>0$, this corner entry is positive or negative depending on the sign of $a_{(n-1)n}$. Since $a_{(n-1)(n-1)}$ and $a_{(n-1)n}$ have opposite signs,  $\widetilde{B}$ is a Type-IV bdsw matrix. In $\widetilde{B}$, the number of negative diagonal entries is still $k$, while the number of positive diagonal entries is $m-1$. By our induction hypothesis, the equivalence $(i)\Leftrightarrow (ii)\Leftrightarrow (iii)$ of the theorem holds for $\widetilde{B}$. To complete the proof, we need to show that this equivalence holds for $A$ as well. We essentially repeat our argument(s) above.
\\
 Suppose $A\in \bfQ$.
Then, $\widetilde{A}\in \bfQ$. Since the last row of $\widetilde{A}$ is nonnegative with $a_{nn}>0$, by Theorem \ref{block matrix Thm}(ii)(a), $\widetilde{B}\in \bfQ$. As the equivalence holds for $\widetilde{B}$,
we must have $(-1)^{k+1}\det \widetilde{B}>0$. By using the Schur determinantal formula $\det A=(\det \widetilde{B})\, a_{nn},$ we have $(-1)^{k+1}\det A>0$. \\
Now suppose $(-1)^{k+1}\det A>0$. Then, as above, $(-1)^{(n-1)+1}\det \widetilde{B}>0$. By our induction hypothesis, $\widetilde{B}\in \bfR_0$ and 
$\deg \widetilde{B}=\pm 1$. Then, by Proposition \ref{degree result} applied to $\widetilde{A}$ with $E=[\frac{1}{a_{nn}}],$
$$\deg \widetilde{A}= \deg \widetilde{B} =\pm 1.$$ By Theorem 6.6.23 in \cite{cottle-pang-stone}, we have
$\deg A= \deg \widetilde{A} =\pm 1.$ Finally, when  $\deg A= \deg \widetilde{A} =\pm 1$, $\deg A$ is nonzero; hence, $A\in \bfQ$. Thus, we have the stated equivalence for $A$.
 \end{proof}
%%%%%%%%%%%%%
Here is an example that shows that $\bfQ$-property is not equivalent to the $\bfR$-property in the case of Type-IV bdsw matrices.

  Consider the following matrix:
$$A=\begin{bmatrix}
        -a & ~~b & ~~0 \\
         ~~0 & -c & ~~d\\
-e & ~~0 & ~~f \end{bmatrix}$$ with $a,b,c,d,e,f>0$ and $\det A<0$. By  Theorem \ref{type4 theorem},  $A \in \bfQ$. However, for any  $d=(d_1,d_2,d_3)^T>0$, $\LCP(A,d)$ has (at least) two solutions: zero and  $(0,d_2/c,0)^T$; hence,  $A\not\in \bfR(d)$. 
\\

\begin{remark}\label{deg of a Q-matrix}
    Let $A$ be a bdsw matrix. By considering results for each type, we see that 
    $$A\in \bfQ\Leftrightarrow A\in \bfR_0\,\,\mbox{and}\,\,\deg A=\pm 1.$$ 
\end{remark}

%%%%%%%%%%%%%%%%%%%%%
\section{A characterization of $\bfQ$-matrices of size $2\times 2$}

In what follows, we provide a complete characterization of the $\bfQ$-property for $2\times 2$ matrices in terms of the sign pattern and determinant. Note that a $2\times 2$ matrix is a bdsw matrix. Recall that in a matrix, $+$ denotes a positive entry, $-$ denotes a negative entry, $\oplus$ denotes a nonnegative entry, and $*$ denotes an arbitrary entry. By abuse of notation, we use the same symbol $A$  to denote a matrix and its sign pattern.

\begin{theorem} \label{2*2 Q characterization}
Let $A$ be a  $2\times 2$ real matrix. Then $A$ has the $\bfQ$-property  if and only if one of the following holds: 
\begin{itemize}
    \item [$(i)$]  $A$ has one of the following sign patterns:
    $$ 
    \begin{bmatrix}
         + & *\\
         \oplus & +
    \end{bmatrix},\,\,\begin{bmatrix}
         + & \oplus\\
         * & +
    \end{bmatrix},\,\,\begin{bmatrix}
         0 & +\\
         - & +
    \end{bmatrix},\,\,\begin{bmatrix}
         + & -\\
         + & 0
    \end{bmatrix};  
    $$
    \item [$(ii)$] $\det A >0$ and $A$ has one of the following sign patterns:
    $$\begin{bmatrix}
        + & - \\
         - & +
    \end{bmatrix},\,\,\begin{bmatrix}
         - & +\\
         - & +
    \end{bmatrix},\,\,\begin{bmatrix}
        + &- \\
        + & -
    \end{bmatrix};$$
\item [$(iii)$] $\det A <0$ and $A$ has the following sign pattern:
    $$\begin{bmatrix}
        - & + \\
         + & -
    \end{bmatrix}.$$
\end{itemize}

\end{theorem}

\begin{proof}
We write a $2 \times 2$ real matrix in the form:
\begin{equation}\label{2by2 matrix}
A=\begin{bmatrix}
        a & b\\
        c & d
    \end{bmatrix}.
    \end{equation}
    {\it As we are characterizing the $\bfQ$-property, we assume that each row of $A$ has a positive entry.} We look at various cases corresponding to the sign of $d$ and, in each case, consider the sign pattern. In this way, we obtain the matrices listed in $(i)$, $(ii)$, and $(iii)$.\\

\noindent{\bf Case 1:} $d>0$.
\begin{itemize}
    \item [$(1a)$] Suppose $A= \begin{bmatrix}
        * & * \\
        \oplus & +
    \end{bmatrix}.$ Then, from Theorem \ref{Bdsw, an1>=0, ann>0}, $A\in \bfQ$ if and only if $A= \begin{bmatrix}
        + &* \\
        \oplus & +
    \end{bmatrix}$. Thus, we have the first matrix in $(i)$. By considering simultaneous row/column changes (which is $JAJ$), we get the second matrix in $(i)$.
    \item [$(1b)$]  Suppose $A= \begin{bmatrix}
        a & b \\
         - & +
    \end{bmatrix}$. Note that both $a$ and $b$ cannot be nonpositive simultaneously; in particular, we cannot have $a<0,b=0$ and $a=0,b<0$. Now, if (both) $a$ and $b$ are nonnegative, then by Theorem \ref{bdsw with kth row nonnegative}, $A\in \bfQ$ if and only if either $A= \begin{bmatrix}
        + & \oplus \\
         - & +
    \end{bmatrix}$ (a special case of the second matrix in $(i)$)  or $A= \begin{bmatrix}
        0 & + \\
         - & +
    \end{bmatrix}$ (the third matrix listed in $(i)$). \\
    On the other hand, if $ab<0$, we have $A= \begin{bmatrix}
        + & - \\
         - & +
    \end{bmatrix}$ or 
    $A= \begin{bmatrix}
        - & + \\
         - & +
    \end{bmatrix}$. 
    Then we can apply either Theorem \ref{type2 bdsw matrix theorem} or Theorem \ref{type4 theorem} (with $k=1$) to see that 
     $A \in \bfQ$ if and only if $\det A>0$. In this way, we get the first two matrices in $(ii)$. The third matrix in $(ii)$ is obtained by (simultaneously) permuting the rows/columns of the second matrix in $(ii)$.
    \end{itemize}

     \noindent{\bf Case 2:} $d=0.$ \\
    
    Then $A=\begin{bmatrix}
        a & b\\
        + & 0 
        \end{bmatrix}.$ By Theorem \ref{Bdsw, an1>0, ann=0}, $A \in \bfQ$ if and only if $A=\begin{bmatrix}
        + & -\\
        + & 0 
        \end{bmatrix}$. Thus, we have the fourth matrix listed in $(i)$; it can also be obtained by simultaneously permuting the rows/columns of the third matrix in $(i)$.\\
        
    \noindent{\bf Case 3:} $d<0$.\\

    Then $A=\begin{bmatrix}
        a & b\\
        + & - 
        \end{bmatrix}.$ Noting that both $a$ and $b$ cannot be nonpositive simultaneously, we consider three cases: 
    \begin{itemize} 
\item [$(3a)$] $A= \begin{bmatrix}
        \oplus & \oplus \\
         + & -
    \end{bmatrix}$. This is a Type-I matrix; by Theorem \ref{bdsw a_n1>0,ann<0}, $A \notin \bfQ$.
    
    \item [$(3b)$] $A= \begin{bmatrix}
        + & - \\
         + & -
    \end{bmatrix}$. This is the third matrix of $(ii)$ considered earlier in $(1b)$. So, $A\in \bfQ$ if and only if $\det A>0$.
    
    \item [$(3c)$] $A= \begin{bmatrix}
        - & + \\
         + & -
    \end{bmatrix}$. By Theorem \ref{type3 theorem} (with $n=2$),  $A\in \bfQ$ if and only if $\det A<0$. This yields the matrix listed in $(iii)$.
\end{itemize}
\end{proof}
{\bf Note:}
As a $2\times 2$ matrix is a bdsw matrix, from Remark \ref{deg of a Q-matrix} we see that
$$A\in \bfQ\Leftrightarrow A\in \bfR_0\,\,\mbox{and}\,\,\deg A=\pm 1.$$

%%%%%%%%%%%%%%%%%%%%%%%
\section{Matrix-based linear transformations on a Euclidean Jordan algebra}
In this section, we extend the results of the previous sections to Euclidean Jordan algebras. First, we recall some basic definitions and introduce notation.
%%%%%%%%%%%%%%%%%%%%%%%%%%%%%%%%
\subsection{Euclidean Jordan algebras}
The material of this subsection can be found in \cite{faraut-koranyi, gowda-sznajder-tao2004}. A Euclidean Jordan algebra is a finite-dimensional real inner product space $(\V, \langle\cdot,\cdot\rangle)$ together with a bilinear product (called the Jordan product) $(x,y)\rightarrow x\circ y$ satisfying the following properties:
\begin{itemize}
\item [$\bullet$] $x\circ y=y\circ x$,
\item [$\bullet$] $x\circ (x^2\circ y)=x^2\circ (x\circ y)$, where $x^2=x\circ x$, and 
\item [$\bullet$] $\langle x\circ y,z\rangle=\langle x,y\circ z\rangle.$
\end{itemize}
In such an algebra, there is the `unit element' $e$ such that $x\circ e=x$ for all $x$.     
In $\V$, $$\V_+=\{x\circ x:\in \V\}$$
is called the {\it symmetric cone} of $\V$. It is a proper cone (i.e., a pointed closed convex cone with nonempty interior), self-dual, and homogeneous.

It is known \cite{faraut-koranyi} that any nonzero Euclidean Jordan algebra is a direct product/sum 
of simple Euclidean Jordan algebras, and every simple Euclidean Jordan algebra is isomorphic to one of five algebras, 
three of which are the algebras of $n\times n$ real/complex/quaternion Hermitian matrices. The other two are: the algebra of $3\times 3$ octonion Hermitian matrices and the Jordan spin algebra.
In the algebras $\Sn$ (of all $n\times n$ real symmetric matrices) and $\Hn$ (of all $n\times n$ complex Hermitian matrices), the Jordan product and the inner product are given, respectively, by 
$$X\circ Y:=\frac{XY+YX}{2}\quad\mbox{and}\quad \langle X,Y\rangle:=\tr(XY),$$
where the trace of a real/complex matrix is the sum of its diagonal entries. 

A nonzero element $c$ in $\V$ is an idempotent if $c^2=c$; it is a primitive idempotent if it is not the sum of two other idempotents. A Jordan frame $\{e_1,e_2,\ldots, e_n\}$ consists of 
primitive idempotents that are mutually orthogonal and with sum equal to the unit element. All Jordan frames in $\V$ have the same number of elements, called the rank of $\V$.
Let the rank of $\V$ be $n$. According to the {\it spectral decomposition 
theorem} \cite{faraut-koranyi}, any element $x\in \V$ has a decomposition
$$x=x_1e_1+x_2e_2+\cdots+x_ne_n,$$
where the real numbers $x_1,x_2,\ldots, x_n$ are (called) the eigenvalues of $x$ and \\
$\{e_1,e_2,\ldots, e_n\}$ is a Jordan frame in $\V$. (An element may have decompositions coming from different Jordan frames, but the eigenvalues remain the same.) 

We use the notation $x\geq 0$ ($x> 0$) when $x\in \V_+$ ($x\in \V_{++}:=$interior of $\V_+$) or, equivalently, all the eigenvalues of $x$ are nonnegative (respectively, positive). 
When $x\geq 0$ with spectral decomposition $x=x_1e_1+x_2e_2+\cdots+x_ne_n$, we define $\sqrt{x}:=\sqrt{x_1}e_1+\sqrt{x_2}e_2+\cdots+\sqrt{x_n}e_n$.
If $x>0$, then $x^{-1}:=(x_1)^{-1}e_1+\cdots+(x_n)^{-1}e_n$.
If $x_1,x_2,\ldots, x_n$ are the eigenvalues of  $x\in \V$, we define the {\it trace }  of $x$ by
$\tr(x):=x_1+x_2+\cdots+x_n.$
It is known that $(x,y)\mapsto \tr(x\circ y)$ defines another inner product on $\V$ that is compatible with the Jordan product. 
{\it We assume that $\V$ carries this inner product,  
that is,
$\langle x,y\rangle=\tr(x\circ y).$}
In this inner product, the norm of any primitive element is one, and so any Jordan frame
in $\V$ is an orthonormal set. Additionally,
$\tr(x)=\langle x,e\rangle \quad\mbox{for all}\,\,x\in \V.$

 It is well-known that  
in $\V$,
\begin{equation}\label{cone nonnegativity}
x,y\geq 0\Rightarrow \langle x,y\rangle \geq 0\quad\mbox{and}\quad
0\neq x\geq 0, \,y>0\Rightarrow \langle x,y\rangle >0.
\end{equation}
%%%%%%%%%%%%
%%%%%%%%%%%%%%%%%%%
\subsection{Symmetric cone LCP}

\begin{definition} Let $L:\V\rightarrow \V$ be a linear transformation and $p\in \V$. Then, the symmetric cone linear complementarity problem, $\LCP(L,\V_+,p)$, is to find $x\in \V$ such that
$$x\in \V_+,\,y:=L(x)+p\in \V_+,\,\mbox{and}\,\,\langle x,y\rangle=0.$$
We write $\SOL(L,\V_+,p)$ for the set of all solutions to $\LCP(L,\V_+,p)$. 
\end{definition}

Notation: In our earlier sections, we consider the standard matrix classes $\bfR_0$, $\bfR$, $\bfQ$, etc. As we are going to define LCP classes in the setting of Euclidean Jordan algebras, we modify the notation: For the LCP classes over $\V$, relative to the symmetric cone $\V_+$, we use the notation(s), $\bfR_0(\V)$, $\bfR(\V)$, $\bfQ(\V)$, etc. In particular, $\bfR_0(\Rn)$ refers to $\bfR_0$, etc.

\begin{definition}
    A linear transformation $\phi:\V\rightarrow \V$ such that $\phi(\V_+)=\V_+$ is called a {\it cone automorphism of $\V$}
    (or an automorphism of $\V_+$). We write $\Aut(\V_+)$ for the set of all such automorphisms.
    \end{definition} 
    Note: Since $\V_+$ has nonempty interior, a cone automorphism is necessarily invertible. A special case of a cone automorphism is an {\it algebra automorphism}: it is an invertible linear transformation $\phi$ on $\V$ that satisfies the condition $\phi(x\circ y)=\phi(x)\circ \phi(y)$ for all $x,y\in \V$. For example, in the algebra $\Rn$, algebra automorphisms are permutation matrices and cone automorphisms are products of permutation matrices and diagonal matrices with positive diagonal. \\
    
    Let $\phi^T$ denote the transpose/adjoint of $\phi$ relative to the inner product in $\V$. Given a linear transformation $L:\V\rightarrow \V$ and $\phi\in \Aut(\V_+)$, we let 
    $$L_\phi:=\phi^TL\phi.$$

    We recall the following result:

    \begin{theorem} \label{L to Lphi}\cite[Theorem 5.1$(a)$]{gowda-sznajder2006} Let $L$ be linear on $\V$ and $\phi\in \Aut(\V_+)$. Then
    \begin{itemize}
    \item $L\in \bfR_0(\V)\Leftrightarrow L_\phi\in \bfR_0(\V)$.
    \item $L\in \bfR(\V)\Leftrightarrow L_\phi\in \bfR(\V)$.
    \item $L\in \bfQ(\V)\Leftrightarrow L_\phi\in \bfQ(\V)$. 
    \end{itemize}
    \end{theorem}

 %%%%%%%%%%%%%%%%%%%%%
\subsection{Matrix-based transformations}
We fix a Jordan frame ${\cal E}=\{e_1,e_2,\ldots,e_n\}$ in $\V$. Given $A=[a_{ij}]\in \MnR$ and $B=[b_{ij}]\in \Sn$, we define a linear transformation $R_{(A,B)}$ as follows: For any $x\in \V$, consider the Peirce decomposition relative to ${\cal E}$ \cite{faraut-koranyi}:
    $$x=\sum_{i=1}^{n}x_ie_i+\sum_{i<j}x_{ij}.$$
  Then,  
  $$R_{(A,B)}(x):=\sum_{i=1}^{n}y_ie_i+\sum_{i<j}b_{ij}x_{ij},$$
    where $y_i:=\sum\limits_{j=1}^{n}a_{ij}x_j$ for $i=1,2,\ldots, n$.

Here are two special cases:
\begin{itemize}
    \item Suppose $B={\bfone} {\bfone} ^T$ (i.e., $b_{ij}=1$ for all $i,j$). Then $R_{(A,B)}$ reduces to $R_A$ studied in \cite{tao-gowda2005}.
    \item Suppose $A$ is a diagonal matrix and the diagonal of $B$ is zero. Let $C=A+B$. Then, $R_{(A,B)}$ reduces to $D_C$ studied in \cite{gowda-tao-sznajder2012}.
    \end{itemize}
    
 In what follows, we let $B=0$ and define $\widehat{A}:=R_{(A,0)}$.

Recall that a rank-one transformation, corresponding to $a,b\in \V$, is given by:
$$(a\otimes b)(x):=\langle b,x\rangle\,a\quad (x\in \V).$$

\begin{definition}
In the given Euclidean Jordan algebra $\V$, we fix a Jordan frame
$\E:=\{e_1,e_2,\ldots, e_n\}$. For any $A=[a_{ij}]\in \MnR$, we define the transformation $\widehat{A}:\V\rightarrow \V$ as follows:
$$
\widehat{A}:= \sum_{i,j=1}^{n}\,a_{ij} e_i\otimes e_j.$$
Explicitly, $$\widehat{A}x=\sum_{i=1}^{n}\Big (\sum_{j=1}^{n}a_{ij}\langle x,e_j\rangle\Big) e_i\quad (x\in \V).$$
\end{definition}

\noindent{\bf Some notation:} As in the above definition, we fix a Jordan frame $\E:=\{e_1,e_2,\ldots, e_n\}$.
For any $r=(r_1,r_2,\ldots, r_n)^T\in \Rn$ and $x\in \V$, we let $$\widehat{r}:=\sum_{i=1}^{n}r_ie_i\quad\mbox{and}\quad
[x]:=\Big ( \langle x,e_1\rangle, \langle x,e_2\rangle,\ldots,\langle x,e_n\rangle\Big )^T.$$
We note:
\begin{equation}\label{diagonal zero}
x\in \V_+, \,[x]=0 \Rightarrow x=0.
\end{equation}
(This is an easy consequence of (\ref{cone nonnegativity}) and the property that $e_1+e_2+\cdots+e_n=e>0$.)

The following are easy to verify:
\begin{equation} \label{widehat remark}
[\widehat{r}]=r,\quad\langle r,s\rangle=\langle \widehat{r},\widehat{s}\rangle,\quad  \widehat{Ar}=\widehat{A}\,\widehat{r},\quad\langle \widehat{A}x,x\rangle=\langle A[x],[x]\rangle.
\end{equation}

   \begin{theorem} \label{embedding theorem}(Embedding theorem) Let $A\in \MnR$ and $q\in \Rn$. Then the following hold:
    \begin{itemize}
    \item [$(i)$] 
         $r\in \SOL(A,q)\Rightarrow \,\widehat{r}\in \SOL(\widehat{A},\V_+,\widehat{q})$.
         \item [$(ii)$]  $x\in \SOL(\widehat{A},\V_+,\widehat{q})\Rightarrow [x]\in \SOL(A,q)$. 
         \end{itemize}
         
    \end{theorem}
    \begin{proof} $(i)$ Let $r\in \SOL(A,q)$. Then the conditions $r\geq 0$, $Ar+q\geq 0$, and $\langle r, Ar+q\rangle =0$  translate to $\widehat{r}\geq 0$, $\widehat{A}\,\widehat{r}+\widehat{q}\geq 0$, and 
    $\langle \widehat{r},\widehat{A}\,\widehat{r}+\widehat{q}\rangle=0.$ We see that $\widehat{r}\in \SOL(\widehat{A},\V_+,\widehat{q}).$\\
    $(ii)$ 
    Now suppose $x$ solves $\LCP(\widehat{A},\V_+,\widehat{q})$. Then,
with $y=
\widehat{A}x+\widehat{q}$, we have
$$[x]\geq 0,\,[y]=A[x]+q,\,\mbox{and}\,\,\langle [x],A[x]+q\rangle=\langle x,y\rangle=0.$$
 Therefore, $[x]$ solves $\LCP(A,q)$.   
    \end{proof}

\begin{theorem} \label{A and Ahat are related} Let $A\in \MnR$. Corresponding to a  Jordan frame $\{e_1,e_2,\ldots,e_n\}$, we consider $\widehat{A}$. Then the following statements hold:
\begin{itemize}
\item [$(a)$] $A$ is copositive (strictly copositive) on $\Rn_+$ $\Leftrightarrow$ $\widehat{A}$ is copositive (strictly copositive) on $\V_+$.
\item [$(b)$] 
$A\in \bfR_0(\Rn)\Leftrightarrow \widehat{A}\in \bfR_0(\V).$
 \item [$(c)$]  $A\in \bfR(\Rn)\Rightarrow \widehat{A}\in \bfR(\V)$. 
\item [$(d)$] $\widehat{A}\in \bfQ(\V) \Rightarrow A\in \bfQ(\Rn)$.
\end{itemize}
\end{theorem}

\begin{proof}
$(a)$: Recall that copositivity of $A$ on $\Rn_+$ means 
$\langle Ax,x\rangle \geq 0$ for all $x\in \Rn_+$ and strict copositivity means $\langle Ax,x\rangle > 0$ for all 
$0\neq x\in \Rn_+$; similar definitions hold for $\widehat{A}$ on $\V_+$. Now the equivalence in $(a)$ follows from (\ref{widehat remark}).\\
$(b)$ Suppose $A\in \bfR_0(\Rn)$ and let $x\in \SOL(\widehat{A},\V_+,0)$. By Theorem \ref{embedding theorem}, $[x]\in \SOL(A,0)$, hence $[x]=0$. As $x\geq 0$, we must have $x=0.$ Thus, $\widehat{A}\in \bfR_0(\V).$\\
To see the reverse implication, suppose $\widehat{A}\in \bfR_0(\V).$ If  $r\in \SOL(A,0)$, then $\widehat{r}\in \SOL(\widehat{A},\V_+,0)=\{0\}$. It follows that $r=0.$\\
$(c)$ Suppose $A\in \bfR(\Rn)$. Then $A\in \bfR_0(\Rn)$ and $\SOL(A,d)=\{0\}$ for some $d >0$ in $\Rn$. By Item $(b)$, $\widehat{A}\in \bfR_0(\V)$. Consider $\widehat{d}$ and 
     let $x\in \SOL(\widehat{A},\V_+,\widehat{d})$. By Proposition \ref{embedding theorem}, $[x]\in \SOL(A,d)=\{0\}$, hence $x=0$. As $\widehat{d} \in \V_{++}$, $\widehat{A}\in \bfR(\V)$.\\
     $(d)$ This follows from Item $(ii)$ in Proposition \ref{embedding theorem}.
\end{proof}

\begin{remark}
Suppose $A$ is a bdsw $\bfQ$-matrix.  If $A$ is of Type-$I$ or $II$, then  $A\in \bfR^*$; consequently, from the above theorem we have $\widehat{A}\in \bfQ(\V)$. What happens when $A$ is of Type-$III$ or $IV$? Then $A\in \bfR_0$ and $\deg A=\pm 1.$ From the above result, we have $\widehat{A}\in \bfR_0(\V)$. In this case, we can define the degree of $\widehat{A}$ relative to $\V_+$ as $\deg \widehat{A}:= \deg (F, \Omega, 0)$, where
$$F(x):=x-(x-\widehat{A}x)^+\quad (x\in \V),$$
$\Omega$ is a bounded open set containing $0$ in $\V$. $($Here, for any $z\in \V$, $z^+$ is the projection of $z$ onto $\V_+$. Note that $F(x)$ is the analog of the function $f(x)=\min\{x, Ax\}$ on $\Rn$ considered earlier while defining the degree of $A$.$)$ 
We do not know how $\deg A$ and $\deg \widehat{A}$ are related. If it happens that $\deg A=\deg \widehat{A}$, we can say: $\widehat{A}\in \bfQ(\V)$ when $A$ is if Type-$III$ or $IV$.

\end{remark}
%%%%%%%%%%%%%%%%%%%
\subsection{Transformations corresponding to nonnegative and rank-one matrices}
 Theorem \ref{A and Ahat are related} shows how the classes $\bfR_0$ and $\bfR$ are related when we go from $A$ to $\widehat{A}$. At this stage, we do not know if the reverse implication in Item $(d)$, namely, 
 \begin{equation}\label{reverse implication}
 A\in \bfQ(\Rn)\Rightarrow \widehat{A}\in \bfQ(\V)
 \end{equation}
 holds generally. In our next two results, we prove (\ref{reverse implication}) when $A$ is either a nonnegative matrix or a rank-one matrix. First, we recall the following consequence of the well-known cone complementarity result due to Karamardian \cite{karamardian}.

\begin{theorem}\label{karamardian}
{\it Let $\V$ be an Euclidean Jordan algebra and $L:\V\rightarrow \V$ be linear. Suppose there is a $d>0$ such that 
zero is the only solution of the problems $\LCP(L,\V_+,0)$ and $\LCP(L,\V_+,d)$. Then,
$L\in \bfQ(\V)$, that is, for all $q\in \V$, $\LCP(L,\V_+,q)$ has a solution. In particular, this conclusion holds if $L$ is strictly copositive on $\V_+$, that is, 
$\langle L(x),x\rangle >0$ for all $0\neq x\in \V_+$.}
\end{theorem}

 \begin{theorem}
Suppose $A\in \MnR$ is a nonnegative matrix. Then, $\widehat{A}\in \bfQ(\V)$ if and only if the diagonal of $A$ is positive. Consequently, $$\widehat{A}\in \bfQ(\V) \Leftrightarrow A\in \bfQ(\Rn).$$
\end{theorem}

\begin{proof} Suppose the diagonal of $A$ is positive. Then, $A$ is strictly copositive; by Item $(a)$ of Theorem \ref{A and Ahat are related}, $\widehat{A}$ is strictly copositive. By the above theorem,  $\widehat{A}\in \bfQ(\V)$.
Now suppose $\widehat{A}$ is in $\bfQ(\V)$. Then for all $r\in \Rn$, $\LCP(\widehat{A},\V_+,\widehat{r})$ has  a solution. By Theorem \ref{embedding theorem}, $\LCP(A,r)$ has a solution for all $r\in \Rn$. Thus, $A$ is a $\bfQ$-matrix. By a well-known result (mentioned in the Introduction), $A$ must have a positive diagonal.
\end{proof}

Our next result deals with a rank-one transformation. It is a generalization of the result mentioned in the Introduction: A rank-one matrix $A=uv^T$ is in $\bfQ$ if and only if $A$ is a positive matrix (that is, either $u$ and $v$ are both positive or both negative).

\begin{theorem}
 Let $a,b\in \V$ and  $L=a\otimes b$. Then the following statements are equivalent:  
\begin{itemize}
\item [$(i)$] $a>0,b>0$ or $a<0,b<0$.
\item [$(ii)$] The implication $0\neq x\geq 0\Rightarrow L(x)>0$ holds.
\item [$(iii)$] $L\in \bfQ(\V).$ 
\end{itemize}

\end{theorem}

We first prove the following:

\begin{lemma}
    For $d\in \V$, let $T:=e\otimes d$, where $e$ denotes the unit element in $\V$. If $T\in \bfQ(\V)$, then $d>0$.
\end{lemma}

\begin{proof}
Suppose $T\in \bfQ(\V)$. We write the spectral decomposition 
$$d=d_1e_1+d_2e_2+\cdots+d_ne_n,$$ where $\{e_1,e_2,\ldots, e_n\}$ is a Jordan frame. From the definition of a Jordan frame,  $e=e_1+e_2+\cdots+e_n$.
Let $A=[a_{ij}]\in \MnR$, where $a_{ij}=d_j$ for all $i,j$. Then corresponding to the Jordan frame $\{e_1,e_2,\ldots, e_n\}$, for any $x\in \V$,
$$\widehat{A}x=\sum_{i=1}^{n}\Big (\sum_{j=1}^{n}a_{ij}\langle x,e_j\rangle\Big) e_i= \sum_{i=1}^{n}\Big \langle x, \sum_{j=1}^{n}d_je_j\Big \rangle e_i=(e\otimes d)(x)=T(x).$$
Then $\widehat{A}=T\in \bfQ(\V)$. From Theorem \ref{A and Ahat are related}, Item $(d)$, $A\in \bfQ(\Rn).$ As $A={\bfone}\,[d]^T$ is a rank-one matrix with ${\bfone} >0$, we must have $[d]>0$ in $\Rn$, that is, $d_i> 0$ for all $i$. This proves that $d>0$ in $\V$.
\end{proof}

\noindent{\bf Remarks.} In the original direct proof, see  \cite{gowda2020}, the above lemma was proved by considering $\LCP(T,\V_+,q)$, where $q=(-1)e_1+0e_2+\cdots+0e_n$ and showing that $d_1>0$ (with a similar proof for other $d_i$).\\

The proof of the theorem is as follows.

\begin{proof} Without loss of generality, we assume that the rank of $\V$ is at least 2.
$(i)\Rightarrow (ii)$: This is obvious as $L(x)=\langle b,x\rangle\,a$.
\\
$(ii)\Rightarrow (i)$: We assume  $(ii)$ so  that for any $0\neq x\geq 0$, $\langle b,x\rangle\,a>0$. Suppose $a>0$: In this case, 
$\langle b,x\rangle >0$ for all $0\neq x\geq 0$. Then, by writing the spectral decomposition of $b$ and taking an appropriate $x$, we see that all the eigenvalues of $b$ are positive; hence $b>0$. Similarly, $b<0$ when $a<0$. Thus we have $(i)$. 
\\
$(ii)\Rightarrow (iii)$: When $(ii)$ holds, $L$ becomes strictly copositive on $\V_+$, that is, $0\neq x\geq 0\Rightarrow \langle L(x),x\rangle >0$. Then by Theorem \ref{karamardian},
$L$  has the $\bfQ$-property. 
\\
$(iii)\Rightarrow (i)$: Suppose $L=a\otimes b\in \bfQ(\V)$. 
With $e$ denoting the unit element of $\V$, LCP$(L,\V_+,-e)$ has a solution, say, $u$. Then, $L(u)-e\geq 0$ implies that $0\neq u\geq 0$ and  
$L(u)\geq e>0$;  so $\langle b,u\rangle \,a>0$. Then, either $a>0$ or $a<0$. 
{\it Supposing $a>0$, we
will show that $b>0$. }(When $a<0$, we have $-a>0$. Since  $L=(-a)\otimes (-b)\in \bfQ(\V)$, we must have $-b>0$, i.e., $b<0$.)\\
We will employ a  standard technique to drive $a$ to $e$ and look at the induced transformation. \\
Now, given that $a>0$, let $c:=\sqrt{a^{-1}}$ and consider the quadratic representation $P_c$ defined by 
$$P_c(x):=2c\circ (c\circ x)-c^2\circ x.$$ 
Because $c$ is invertible, from the well-known properties of quadratic representation, we see that $P_c$ is self-adjoint and invertible,  
$$(P_c)^{-1}=P_{c^{-1}}, \,P_c(\V_+)=\V_+,\,\mbox{and}\,P_c(a)=e.$$
In particular, $\phi:=P_c\in \Aut(\V_+).$
Noting that $\phi$ is self-adjoint, we define a new linear transformation $T$ on $\V$ by 
$$T:=L_\phi=P_cLP_c.$$
As $L\in \bfQ(\V)$, from Theorem \ref{L to Lphi}, $T\in \bfQ(\V)$. 
Now, define $d:=P_c(b)$ so that 
$$T(x)=P_cLP_c(x)=P_c\big (\langle P_c(x),b\rangle \,a\big)=\langle x,P_c(b)\rangle\,P_c(a)=\langle x,d\rangle\,e\quad \mbox{for all}\,\,x\in \V.$$ 
This means that 
$$T=e\otimes d.$$
As $T\in \bfQ(\V)$, from the above lemma, $d>0$. Since $P_c$ is a cone automorphism, it maps the interior of $\V_+$ onto itself. Thus,
$$P_c(b)=d>0\Rightarrow b=P_{c^{-1}}(d)>0.$$   So we have proved that $(iii)\Rightarrow (i)$. This completes the proof of the theorem. 
\end{proof} 
%%%%%%%%%%%%%%%

\section{Concluding remarks and future work} In this paper, we characterized the $\bfQ$-property of banded matrices, specifically concentrating on triangular and bidiagonal southwest matrices. This study motivates the following questions:
$(i)$ What are the $\bfQ_0$-properties of banded matrices?
$(ii)$ What are the LCP properties of other special types of matrices, such as tridiagonal matrices, circular matrices, Toeplitz matrices, etc.?  $(iii)$ What are the graph theory implications of LCP properties? $(iv)$ In connection with Euclidean Jordan algebras, how are $\deg A$ and $\deg \widehat{A}$ related? These questions will be taken up in a future study.

\section{Acknowledgments:} 
Samapti Pratihar acknowledges the support provided by the Prime Minister’s Research Fellowship (PMRF), Ministry of Education, Government of India (Project Number: SB22231580MAPMRF008118) for carrying out this work. She also thanks the Office of Global Engagement, Indian Institute of Technology Madras, for financial assistance to visit the University of Maryland, Baltimore County, through their IIE program, as well as the Department of Mathematics and Statistics, University of Maryland, Baltimore County, for partial financial assistance and excellent hospitality during her visit.

%%%%%%%%%%%%%%%%%%%%%%%%%%%%%%%%%%%%%%%%%%%%%%%%%%%%%%%%%%%%%


\begin{thebibliography}{10}

\bibitem{berman-plemmons} Berman, A., Plemmons, R.J., Nonnegative Matrices in the Mathematical Sciences. SIAM, Philadelphia (1994).

\bibitem{cottle-pang-stone} Cottle, R.W., Pang, J.S., Stone, R.E., The Linear Complementarity Problem. SIAM, Philadelphia (2009).

\bibitem{faraut-koranyi} Faraut, J., Kor{\'a}nyi, A.,  Analysis on Symmetric Cones. Oxford University Press, Oxford (1994).

\bibitem{facchinei-pang} Facchinei, F., Pang, J.-S., Finite-dimensional Variational Inequalities and Complementarity Problems. Vol. I, Springer, New York (2003).

\bibitem{gowda-degree1993} Gowda, M.S., Applications of degree theory to linear complementarity problems. Math. Oper. Res. {\bf 18}, 868--879 (1993).

\bibitem{gowda2020} Gowda, M.S., A characterization of $Q$-property for rank-one linear transformations on Euclidean Jordan algebras. UMBC Online Seminar, Sept. 24, 2020, https://userpages.umbc.edu/$\sim$gowda/presentations/index.html.

\bibitem{gowda-song1999} Gowda, M.S., Song, Y., On semidefinite linear complementarity problems. Math. Prog., Series A. {\bf 88}, 575--587 (2000).

\bibitem{gowda-sznajder2006} Gowda, M.S., Sznajder, R., Automorphism invariance of P and GUS properties of linear transformations on Euclidean Jordan algebras. Math. Oper. Res. {\bf 31}, 109--123 (2006).

\bibitem{gowda-sznajder-tao2004} Gowda, M.S., Sznajder, R., Tao. J., Some P-properties for linear transformations on Euclidean Jordan algebras.  Linear Algebra Appl. {\bf 393}, 203--232 (2004).

\bibitem{gowda-tao-sznajder2012} Gowda, M.S., Tao, J., Sznajder, R., Complementarity properties of Peirce-diagonalizable linear transformations on Euclidean Jordan algebras. Optim. Methods Software. {\bf 27}, 719--733 (2012). 

\bibitem{karamardian} Karamardian, S., An existence theorem for the complementarity problem. J. Optim. Theory Appl. {\bf 19}, 227--232 (1976).

\bibitem{higham} Higham, N.J., The power of bidiagonal matrices. Electron. J. Linear Algebra. {\bf 40}, 453--474 (2024).

\bibitem{murty} Murty, K.G., On the number of solutions to the complementarity problem and
spanning properties of complementary cones. Linear Algebra Appl. {\bf 5}, 65--108 (1972). 

\bibitem{pratihar-sivakumar2024} Pratihar, S., Sivakumar, K.C., Inverse $Z$-matrices with the bidiagonal southwest structure. Contemp. Math., 833, Amer. Math. Soc., [Providence], RI, 211--225 (2026). 
 
\bibitem{sivakumar-sushmitha-tsatsomeros2022} Sivakumar, K.C., Sushmitha, P., Tsatsomeros, M., 
$Q_{\#}$-matrices and $Q_{\dag}$-matrices: two extensions of the $Q$-matrix concept. Linear and Multilinear Algebra. {\bf 70}, 6947--6964 (2022).


%\bibitem{sivakumar-parameswaran-wendler2021} Sivakumar, K.C., Parameswaran, S. and Wendler, M., 2021. Karamardian Matrices: An Analogue of Q-matrices. The Electronic Journal of Linear Algebra, 37, pp.127-155.


 \bibitem{tao-gowda2005} Tao, J., Gowda, M.S.,  Some P-Properties for nonlinear transformations on
 Euclidean Jordan algebras. Math. Oper. Res. {\bf 30}, 985--1004 (2005).
\end{thebibliography}
\end{document}